\documentclass[12pt]{amsart}

\usepackage{amsmath, amscd, graphicx, latexsym, hyperref, rlepsf, times, gensymb, soul}
\usepackage{mathtools}
\usepackage{xcolor,enumerate}
\usepackage{blkarray}
\usepackage{color}
\usepackage{tikz}
\newcommand*\circled[1]{\tikz[baseline=(char.base)]{
            \node[shape=circle,draw,inner sep=2pt] (char) {#1};}}

\newcommand{\blue}[1]{\textcolor{blue}{\bf #1}}

\oddsidemargin .25in \theoremstyle{plain}

\newtheorem{Thm}{Theorem}

\newtheorem{Lem}[Thm]{Lemma}

\newtheorem{Cor}[Thm]{Corollary}

\newtheorem{Prop}[Thm]{Proposition}

\newtheorem{Def}[Thm]{Definition}

\newtheorem{Rem}[Thm]{Remark}

\newcommand{\bfc}{{\mathbb{C}}}

\newcommand{\bfr}{{\mathbb{R}}}

\newcommand{\id}{\operatorname{id}}

\def\cpb{\hbox{$\overline{{\Bbb C}P^2}$}}

\def\a{\alpha}
\def\b{\beta}
\def\d{\delta} \def\D{\Delta}
\def\g{\gamma}
\def\G{\Gamma}

\begin{document}

\title{Unbraided wiring diagrams for Stein fillings of lens spaces}

\author{Mohan Bhupal and Burak Ozbagci}

\address{Department of Mathematics,  METU, Ankara, Turkey, \newline bhupal@metu.edu.tr}
\address{Department of Mathematics, Ko\c{c} University, Istanbul, Turkey
\newline bozbagci@ku.edu.tr}


\begin{abstract}

In a previous work \cite{bo}, we constructed a planar Lefschetz fibration on each Stein filling of any lens space equipped with its canonical contact structure. Here we  describe an algorithm to draw an {\em unbraided} wiring diagram whose associated planar Lefschetz fibration obtained by the method of Plamenevskaya and  Starkston \cite{ps}, where the lens space with its  canonical contact structure is viewed as the contact link of a cyclic quotient singularity,  is equivalent to the Lefschetz fibration we constructed on each Stein filling of the lens space at hand. Coupled with the work of  Plamenevskaya and  Starkston, we obtain the following result as a corollary:  The wiring diagram we describe can be extended to an arrangement of symplectic graphical disks  in $\bfc^2$ with marked points, including all the intersection points of these disks,  so that by removing the proper transforms of these disks from the blowup of $\bfc^2$ along those marked points one recovers the Stein filling along with the Lefschetz fibration. Moreover, the arrangement is related to the decorated plane curve germ representing the cyclic quotient singularity by a smooth graphical homotopy.

As another application, we set up an explicit bijection between the Stein fillings of any lens space with its  canonical contact structure, and the Milnor fibres of the corresponding cyclic quotient singularity, which was first obtained by  N\'{e}methi and Popescu-Pampu \cite{npp}, using different methods.
\end{abstract}

\maketitle

\section{Introduction}

In their recent work, Plamenevskaya and  Starkston   \cite{ps} showed that every Stein filling of the link of  a rational surface singularity with reduced fundamental cycle,  equipped with its canonical contact structure, can be obtained from a configuration of symplectic graphical disks in $\mathbb{C}^2$ with marked points including all the intersection points of these disks, by removing the union of the proper transforms of these disks  from the blowup of $\bfc^2$ at the marked points. Their  purely topological proof relies on a theorem of Wendl \cite{w}, which implies that each Stein filling of the contact singularity link of the type above admits a planar Lefschetz fibration over $D^2$, since the contact link itself is supported by a planar open book (\cite{ab},\cite{nt}), and moreover, the Lefschetz fibration corresponds to a positive factorization of the monodromy of this open book. In their proof,  Plamenevskaya and  Starkston developed a method to reverse-engineer a braided wiring diagram producing any such factorization, and then extended this diagram to an arrangement of symplectic graphical disks which, in turn, gives the Stein filling along with the Lefschetz fibration via the method described above (see Section~\ref{sec: bwd} for more details of their construction).

In the present  paper, we focus on cyclic quotient singularities---a subclass of rational surface singularities with reduced fundamental cycle. It is well-known that the oriented link of any cyclic quotient singularity is orientation preserving diffeomorphic to a lens space $L(p,q)$.
Let $\xi_{can}$  denote the canonical contact structure on the singularity link $L(p,q)$. In \cite{l}, Lisca showed that any minimal symplectic filling of the contact $3$-manifold $(L(p,q), \xi_{can})$, is orientation-preserving diffeomorphic to some $4$-manifold $W_{p,q}(\textbf{n})$ (see Section~\ref{sec: lens} for its definition). Moreover, he showed that $W_{p,q}(\textbf{n})$  is in fact a Stein filling of  $(L(p,q), \xi_{can})$. 

In \cite{bo}, the authors constructed a {\em planar} Lefschetz fibration $W_{p,q}(\textbf{n}) \to D^2$, by explicitly describing the ordered set of vanishing cycles on a disk with holes. In this paper, our primary goal is to describe an algorithm to draw a wiring diagram, which turns out  to be {\em unbraided},
that corresponds to each of these planar Lefschetz fibrations via the work of Plamenevskaya and  Starkston.

\begin{Thm}\label{thm: main} There is an algorithm to draw an explicit unbraided wiring diagram whose associated planar Lefschetz fibration obtained by the method of Plamenevskaya and  Starkston \cite{ps} is equivalent to the planar Lefschetz fibration $W_{p,q}(\textbf{n}) \to D^2$  constructed by the authors in \cite{bo}.
\end{Thm}

As we mentioned in the first paragraph, Plamenevskaya and  Starkston described an algorithm (see \cite[Section 5.4]{ps}) to obtain a braided wiring diagram from the  ordered set of vanishing cycles of a planar Lefschetz fibration, by reverse-engineering. Their algorithm involves many choices (see \cite[Remark 5.7]{ps}) and although we do not rely on their reverse-engineering algorithm here, we show {\em indirectly} that by appropriate choices, one can obtain ``unbraided" wiring diagrams, which means that all the braids in the diagram can be chosen to be the identity.

The article of Plamenevskaya and Starkston was admittedly inspired by the work of de Jong and van Straten \cite{djvs}, who  studied the Milnor fibres and deformation theory of sandwiched singularities---which includes rational surface singularities with reduced fundamental cycle. In their work, deformation theory of a surface singularity in the given class is reduced to deformations of the germ of a reducible plane curve representing the singularity.

In particular, to any cyclic quotient singularity germ, de Jong and van Straten associate a decorated germ of a reduced plane curve singularity $\mathcal C=C_1\cup\cdots\cup C_n\subset (\bfc^2,0)$ with {\em smooth} irreducible branches, where the decoration on each $C_i$ is a certain positive integer, which we omit here from the notation for simplicity.  The outcome of de Jong and van Straten's construction is that there is a bijection between one-parameter deformations of the cyclic quotient singularity and ``picture deformations'' of $\mathcal C $ representing that singularity (see Section~\ref{sec: bwd} for more details of their construction).

Moreover, Plamenevskaya and  Starkston \cite[Proposition 5.5]{ps} extend any given braided wiring diagram (viewed as a collection of intersecting curves in $\bfr \times \bfc$) to a collection of symplectic disks in $\bfc \times \bfc$. Consequently, as a corollary to Theorem~\ref{thm: main}, we obtain the following result coupled with \cite[Proposition 5.8]{ps}.

\begin{Cor} \label{cor: arrangement} For each Stein filling $W_{p,q}(\textbf{n})$ of the contact $3$-manifold $(L(p,q), \xi_{can})$, there is an explicit  collection of symplectic graphical disks $\G_1, \ldots , \G_n$ in $\bfc^2$, with marked points $p_1, \ldots, p_m \subset \bigcup_i \G_i$, which include all the intersection points of these disks, so that by removing the union of the proper transforms $\widetilde{\G}_1,  \ldots, \widetilde{\G}_n$ of $\G_1,  \ldots, \G_n$, from the blowup of $\bfc^2$ at these marked points we obtain $W_{p,q}(\textbf{n})$ along with the Lefchetz fibration mentioned in Theorem~\ref{thm: main}. Moreover, the collection of symplectic graphical disks $\G_1, \ldots , \G_n$ is related to the decorated plane curve germ $\mathcal C=C_1\cup\cdots\cup C_n\subset (\bfc^2,0)$  representing the cyclic quotient surface singularity at hand by a smooth graphical homotopy.
\end{Cor}

 For each Stein filling $W_{p,q}(\textbf{n})$ of  $(L(p,q), \xi_{can})$, the symplectic graphical disk arrangement $\G:= \G_1 \cup  \cdots \cup \G_n $ with marked points $p_1, \ldots, p_m \subset \G$ in $\bfc^2$ described in Corollary~\ref{cor: arrangement} determines immediately the $m \times n$ {\em  incidence matrix} $\mathcal{I}(\G, \{p_j\})$, defined so that its $(i,j)$-entry is $1$ if $p_j \in \G_i$, and $0$ otherwise. Since there is no canonical labelling of the points $p_j$, in general, the incidence matrix is only defined up to permutation of the columns.

\begin{Cor} \label{cor: incidence}  For each Stein filling $W_{p,q}(\textbf{n})$ of  $(L(p,q), \xi_{can})$, there is an iterative  algorithm to obtain the incidence matrix  $\mathcal{I}(\G, \{p_j\})$ for the symplectic graphical disk arrangement $\G:= \G_1 \cup  \cdots \cup \G_n $ with marked points $p_1, \ldots, p_m \subset \G$ in $\bfc^2$ described in Corollary~\ref{cor: arrangement}. \end{Cor}

As a matter of fact, one can read off  the incidence matrix  directly from the wiring diagram from which the symplectic disk configuration arises. As explained in \cite[Section 6]{ps} and \cite[Section 5]{djvs}, the incidence matrix  $\mathcal{I}(\G, \{p_j\})$ determines the fundamental group, the integral homology and the intersection form of $W_{p,q}(\textbf{n})$, as well as the first Chern class of the Stein structure on $W_{p,q}(\textbf{n})$.  It is worth mentioning that the incidence matrices can be used to distinguish the diffeomorphism types of the fillings (see \cite{npp}) as an application of the algorithm in Corollary~\ref{cor: incidence}.

Note that each Milnor fibre of the cyclic quotient singularity is a Stein filling of its boundary---which is the link $L(p,q)$ of the singularity, equipped with its canonical contact structure
$\xi_{can}$.  In \cite{npp},   N\'{e}methi and Popescu-Pampu showed that  there is  an explicit  one-to-one correspondence between Stein fillings of $(L(p,q), \xi_{can})$ and Milnor fibres of the corresponding cyclic quotient singularity, proving in particular a conjecture of Lisca \cite{l}.
As another application of Theorem~\ref{thm: main}, we obtain an alternate proof of their result formulated as Corollary~\ref{cor: equiv}.
We say that two (smooth) disk arrangements $(\G, \{p_j\})$ and $(\G', \{p'_j\})$  in $\mathbb{C}^2$ are {\em combinatorially equivalent} if their incidence matrices coincide up to permutation of columns, i.e. up to relabelling of the marked points.

\begin{Cor} \label{cor: equiv}  For each Stein filling $W_{p,q}(\textbf{n})$ of  $(L(p,q), \xi_{can})$, the arrangement of symplectic graphical disks $(\G, \{p_j\})$, described in  Corollary~\ref{cor: arrangement},  is combinatorially equivalent  to the arrangement of the smooth branches of a picture  deformation  of the decorated plane curve  germ $\mathcal C$ representing the corresponding cyclic quotient surface singularity, described by de Jong and van Straten \cite{djvs}.  This equivalence gives an explicit  bijection between Stein fillings of $(L(p,q), \xi_{can})$ and Milnor fibres of the corresponding cyclic quotient singularity. \end{Cor}

In other words, the wiring diagrams we obtain via Theorem~\ref{thm: main} can be viewed as picture deformations of the decorated plane curve germ representing the associated cyclic quotient singularity. In Figure~\ref{fig: steintomilnor} below, we summarized the correspondence that takes each Stein filling of $(L(p,q), \xi_{can})$ given by Lisca to the Milnor fibre of the associated cyclic quotient singularity described by de Jong and van Straten via a picture deformation of the decorated plane curve representing the singularity.

\begin{figure}[ht]  \relabelbox \small {\epsfxsize=4.5in
\centerline{\epsfbox{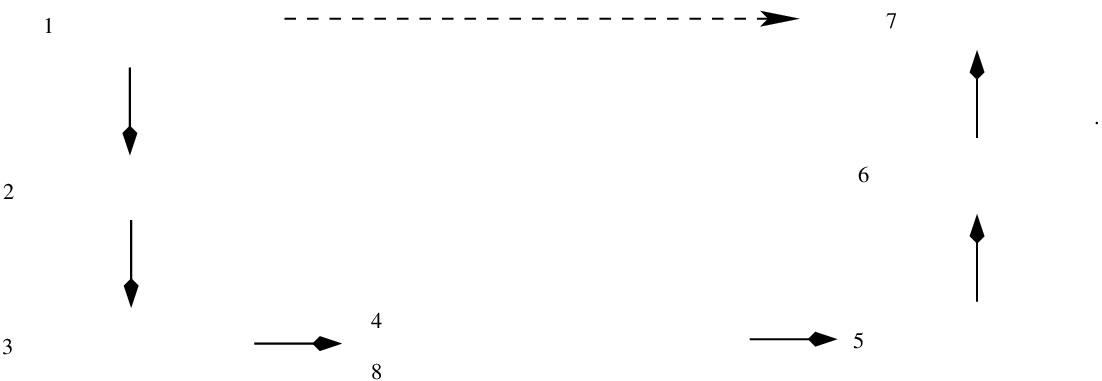}}}
\relabel{1}{Stein filling }  \relabel{2}{Lefschetz fibration} \relabel{3}{Wiring diagram}
\relabel{4}{Configuration of symp.} \relabel{8}{graphical disks in $\mathbb{C}^2$} \relabel{5}{Incidence matrix} \relabel{6}{picture deformation} \relabel{7}{Milnor fibre}
 \endrelabelbox
\caption{From Lisca to de Jong and van Straten  }\label{fig: steintomilnor} \end{figure}

It is well-known that for any contact singularity link $(L(p,q), \xi_{can})$, the Milnor fibre of the Artin smoothing component of the corresponding cyclic quotient singularity is diffeomorphic to the  Stein filling $W_{(p, q)}((1,2,\ldots,2,1))$,   which is Stein deformation equivalent to the one obtained by deforming the symplectic structure on the minimal resolution (see \cite{bd}) of the singularity. In addition, according to \cite{djvs}, there is a canonical picture deformation, called the {\em Scott deformation},  of the decorated plane curve germ which corresponds to the Artin smoothing. As a final application of Theorem~\ref{thm: main}, we  obtain a wiring diagram that represents the combinatorial equivalence class of the Scott deformation---which illustrates the correspondence in Figure~\ref{fig: steintomilnor} between the Stein filling $W_{(p, q)}((1,2,\ldots,2,1))$ and the Milnor fibre of the Artin smoothing component. 

\begin{Cor} \label{cor: scott} For the Stein filling  $W_{(p, q)}((1,2,\ldots,2,1))$ of  $(L(p,q), \xi_{can})$, the arrangement of symplectic disks given in  Corollary~\ref{cor: arrangement}, arising from the wiring diagram described in Theorem~\ref{thm: main},  is combinatorially equivalent  to the Scott deformation of a decorated plane curve representing the singularity.
\end{Cor}

\section{Rational singularities with reduced fundamental cycle, picture deformations, and  braided wiring diagrams} \label{sec: bwd}

In \cite{npp}, N\'{e}methi and Popescu-Pampu showed that there is an
explicit bijective correspondence between Stein fillings of the link of a cyclic quotient singularity and Milnor fibres of smoothing components
of the given singularity, as conjectured by Lisca \cite{l}. As cyclic quotient singularities are examples of sandwiched singularites, de Jong and van Straten's picture deformation construction (see \cite{djvs}) can be used to describe these Milnor fibres. We give a brief description of the construction of de Jong and van Straten here.
As the theory is easier to describe in the case of rational singularities with reduced fundamental cycle, a class which
contains cyclic quotient singularities, we will restrict attention to these.

Let $(X,x)$ be the germ of a rational singularity with reduced fundamental cycle.
De Jong and van Straten associate to $(X,x)$ a, possibly nonunique, decorated germ of a reduced plane curve singularity $\mathcal C=C_1\cup\cdots\cup C_n\subset (\bfc^2,0)$ with \emph{smooth} irreducible branches. Each such singularity can be resolved by a finite sequence of blowups.
For each branch $C_i$, let $m_i$ denote the number of times $C_i$, or its proper transform,
is blown up in the minimal resolution. For example, if $\mathcal C$ consists of a collection of curves intersecting
(pairwise) transversally at $0$, then $m_i=1$ for all $i$.
The decoration on $\mathcal C$ consists of an $n$-tuple $l=(l_1,\dots,l_n)$ of positive integers such that $l_i\geq m_i$
for each $i$. Given such
a decorated curve $(\mathcal C,l)$, one can recover the corresponding surface singularity as follows: Take the minimal
embedded resolution of $\mathcal C$ and iteratively blow up the proper transform of $C_i$ $(l_i-m_i)$-times on the preimage
of $0$ for each $i$. Under the condition that $l_i>m_i$ for each $i$, the set of exceptional curves that do not meet
the proper transform of $\mathcal C$ will be connected. Collapsing them then gives the corresponding surface singularity.

Given a decorated curve $(\mathcal C,l)$, let $\tilde{\mathcal C}=\tilde C_1\cup\cdots\cup\tilde C_n$ denote the normalization of $\mathcal C$ (which in our present situation is just the disjoint union of the irreducible components of $\mathcal C$).
Geometrically, one may think of the decoration $l$ as a collection of $l_i$ marked points on $\tilde C_i$, for each $i$, all concentrated on
the preimage of the singular point.
The outcome of de Jong and van Straten's construction is that there is a one-to-one correspondence between one-parameter deformations of $(X,x)$ and ``picture deformations'' of $(\mathcal C,l)$. Roughly speaking, a \emph{picture deformation} of $(\mathcal C,l)$
consists of a $\delta$-constant deformation $\mathcal C^s=C_1^s\cup\cdots\cup C_n^s$ of $\mathcal C$, which in the present situation means that the branches of $\mathcal C$ are deformed separately
and not allowed to merge, together with a redistribution $l^s$ of the marked points so that we have exactly $l_i$ marked points on $\tilde C^s_i$ for each $i$, where $\tilde C^s_1,\ldots,\tilde C^s_n$ denote the irreducible components of the normalization $\tilde{\mathcal C}^s$ of $\mathcal C^s$. Here $\mathcal C^0=\mathcal C$ and we require that for $s\neq 0$ the only singularities of $\mathcal C^s$ are ordinary
$k$-tuple points, for various $k$, that is, transversal intersections of $k$ smooth branches, and that each such multiple point
is marked. There may be additional ``free'' marked points on the branches of $\mathcal C^s$. The Milnor fibre of the smoothing
associated to $(\mathcal C^s,l^s)$  can then be constructed by blowing up all the marked points, taking the complement
of the proper transforms of $C_1^s,\ldots,C_n^s$ and smoothing corners. Here the Milnor fibre will be noncompact, but by
working in a small ball centered at the origin in $\bfc^2$ we can obtain compact Milnor fibres.

The topological information from picture deformations can be conveniently extracted  by using the notion of braided wiring diagrams. These were introduced by
Cohen and Suciu \cite{cs} in their study of complex hyperplane arrangements and have been used fruitfully by Plamenevskaya and Starkston \cite{ps} in
their investigation of unexpected Stein fillings in the case of rational surface singularities with
reduced fundamental cycle. We briefly describe these next.

A \emph{braided wiring diagram} is a collection of curves $f_i\colon [0,1]\to\bfr\times\bfc$ for $1\leq i\leq n$, called \emph{wires},
such that $f_i(t)\in\{t\}\times\bfc$. At finitely many interior points $t_1,\ldots,t_m$, a subcollection of the wires may intersect with the
remaining being disjoint, but at each such point
the wires intersecting are assumed to have distinct tangent lines. We will make the further assumption that there is
a number $\varepsilon> 0$ such that the positions of the wires above the points $0,1$ and $t_i\pm\varepsilon$ take the same given values
in $\bfr\subset\bfc$ and the restriction of each wire $f_j$ to $(t_i-\varepsilon,t_i+\varepsilon)$ is linear. Any
braided wiring diagram can be made to satisfy this assumption by a homotopy of braided wiring diagrams.
Then the portions of the braided wiring diagram between $t_i-\varepsilon$ and $t_i+\varepsilon$ can be
specified by declaring which adjacent wires intersect and on the complementary intervals the wires may be braided.  Moreover, any wiring diagram will be presented by its projection onto $\bfr \times \bfr \subset \bfr \times \bfc$, where the second $\bfr$ is the real part of $\bfc$.

We now describe how to obtain a braided wiring diagram from a picture deformation $(\mathcal C^s,l^s)$. By choosing
coordinates of $\bfc^2$ generically, we may assume that each $C_i^0$ is graphical, that is,
$C_i^0=\{(x,y)\in\bfc^2\,|\,x\in D, y=f_i(x)\}$ for some complex function $f_i$, where $D$ is a small disk in $\bfc$
centered at $0$. For $s>0$ sufficiently small, it follows that each $C_i^s$ is graphical. Let $\eta_1,\ldots,\eta_m$
denote the images of the intersection points of $C_1^s,\ldots,C_n^s$ under the map $\pi_x \colon\bfc^2\to\bfc$
given by projecting onto the first coordinate and choose a smooth curve $\gamma\colon [0,1]\to D$ whose interior passes through these points such that $\gamma^\prime(t)$ has
nonpositive real part for all $t$. Then $(C_1^s\cup\cdots\cup C_n^s)\cap\pi_x^{-1}(\gamma)$ is a braided wiring diagram.

Next we review how Plamenevskaya and Starkston constructed planar Lefschetz fibrations based on a configuration of smooth disks in $\bfc^2$; see \cite[Lemma 3.2]{ps}. Let $\G_1,  \ldots, \G_n$ be smooth disks in $\bfc^2$ which are graphical with respect to the projection $\pi_{x}$.   Assume that whenever two or more of these disks meet at a point, they intersect transversally and positively   with respect
to the orientation on the graph $\G_i$ induced from the natural orientation on $\bfc$. Let $p_1, \ldots, p_m$  be the marked points
on $\bigcup_i \G_i$ which include all the intersection points, and let $ \Pi\colon \bfc^2 \# m \cpb
 \to \bfc^2 $  be the blow-up at the
points $p_1, \ldots, p_m$. If $\widetilde{\G}_1,  \ldots, \widetilde{\G}_n$
denote the proper transforms of $\G_1,  \ldots, \G_n$,
then $$\pi_{x} \circ \Pi\colon \bfc^2 \# m \cpb \setminus  (\widetilde{\G}_1\cup  \cdots\cup \widetilde{\G}_n) \to \bfc$$ is a {\em Lefschetz fibration} whose regular fibres are
punctured planes, where each puncture corresponds to a component $\widetilde{\G}_i$. There is one vanishing cycle for
each point $p_j$, which is a curve in the fibre enclosing the punctures that correspond to the components $\G_i$
passing through $p_j$.

Moreover, restricting to an appropriate  Milnor ball in $\bfc^2$ that contains all the points $p_1, \ldots, p_m$ one obtains a Lefschetz fibration whose fibre is a disk with holes, where the holes correspond
to the components $\G_i$ and the vanishing cycles correspond to the points $p_j$ in the same way as described above.
Furthermore, if the {\em curvettas} $C_1^s,\ldots,C_n^s$
with marked points are the result of a picture deformation of a germ associated to a surface singularity, the Lefschetz fibration constructed as above is compatible
with the complex structure on the Milnor fibre of the corresponding smoothing.

\subsection{From wiring diagrams to planar Lefschetz fibrations} \label{subsec: wiringtoPLF} Here we outline the method of Plamenevskaya and Starkston that gives a set of ordered vanishing cycles associated to any braided wiring diagram, which in turn determines a planar Lefschetz fibration on the associated Stein fillings; see \cite[Section 5.2]{ps}.
In this paper, we will only deal with wiring diagrams without any braids and we will call them {\em unbraided} wiring diagrams. In the following, we describe their method for the case of unbraided wiring diagrams. {\em We should emphasize that our conventions will be different from those of \cite{ps}, for the purposes of this paper.} We denote the  marked points (consisting of intersection points and free points) by $x_i$, and enumerate them according to their geometric position from right to left, as illustrated in Figure~\ref{fig: wire}.

\begin{figure}[ht]  \relabelbox \small {\epsfxsize=2in
\centerline{\epsfbox{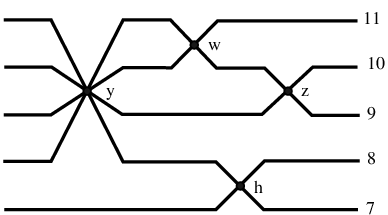}}}
 \relabel{z}{$x_1$}  \relabel{w}{$x_3$} \relabel{y}{$x_4$}  \relabel{h}{$x_2$} \endrelabelbox
\caption{An example of an {\em unbraided} wiring diagram without any free points.} \label{fig: wire} \end{figure}

For any positive integer $r$, let $D_r$ denote the disk with $r$ holes. Throughout this section we assume that the holes are aligned vertically in $D_r$.  For each marked point $x_s$ in a given wiring diagram with $k$ wires,  there is a convex curve $\d (x_s)$ in $D_k$ enclosing a certain set of adjacent holes, which is determined as follows.

\begin{Def} \label{def: convex} {\em (Convex curve assigned to a marked point)} Suppose that the marked point $x_s$  is a simultaneous intersection point of some geometrically consecutive wires in a given wiring  diagram. The convex curve $\d (x_s)$ encircling the adjacent holes whose geometric order from the top in $D_k$ coincides with the local geometric order of the wires simultaneously intersecting at that marked point is called the convex curve assigned to  $x_s$. If $x_s$ is a free marked point on a single wire, then the convex curve $\d (x_s)$ assigned to $x_s$ is the curve which is parallel to a single interior boundary component of $D_k$ whose order from the top coincides with the local geometric order of the wire. \end{Def}

For example, in Figure~\ref{fig: wire},  the geometrically top four wires intersect at the marked point $x_4$;  the geometrically top two wires intersect at the marked point $x_3$;  the geometrically bottom two wires intersect at the marked point $x_2$ and  the geometrically second and third wires intersect at the marked point $x_1$. It follows that the convex curves $\d (x_4), \d (x_3), \d (x_2), \d (x_1)$ depicted  in Figure~\ref{fig: ex9b} are assigned to the marked points $x_4, x_3, x_2, x_1$, respectively, in Figure~\ref{fig: wire}.

\begin{figure}[ht]  \relabelbox \small {\epsfxsize=1.8in
\centerline{\epsfbox{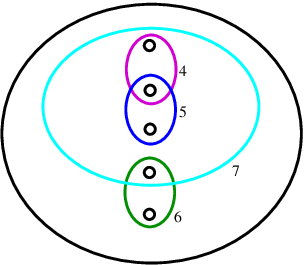}}}
 \relabel{4}{$\d (x_3)$}  \relabel{5}{$\d (x_1)$}  \relabel{6}{$\d (x_2)$} \relabel{7}{$\d (x_4)$}
 \endrelabelbox
\caption{Convex curves in $D_5$ assigned to the marked points in Figure~\ref{fig: wire}.} \label{fig: ex9b} \end{figure}

For each marked point $x_s$ in the wiring diagram,  there is a counterclockwise half-twist $\D (x_s)\colon D_k \to D_k$, which is determined as follows.

\begin{Def} \label{def: twist}  {\em (Counterclockwise half-twist corresponding to a marked point)}  The counterclockwise half-twist $\D (x_s)$ along the subdisk in $D_k$ enclosed by the convex curve $\d (x_s)$  is called the counterclockwise half-twist corresponding to $x_s$. \end{Def}

Suppose that a wiring diagram has $k$ wires and  $r$ marked points $x_r, x_{r-1}, \ldots, x_1$, reading from left to right. According to \cite{ps},  for each $1 \leq s \leq r$, there is a vanishing cycle $V(x_s)$ in $D_k$ associated to the marked point $x_s$, which is determined as follows.

 \begin{Def} \label{def: cycle} {\em (Vanishing cycle associated to a marked point)} For each $2 \leq s \leq r$, the vanishing cycle $V(x_s)$ associated to the marked point $x_s$ is the curve  in $D_k$ given as $$ \D(x_{1}) \circ \cdots \circ  \D(x_{s-1}) (\d(x_s)),$$ and $V(x_1) =\d (x_1)$.   \end{Def}

For example, the vanishing cycles for the marked points in   Figure~\ref{fig: wire} are calculated as follows. The curve  $V(x_4) = \D(x_{1}) \circ \D(x_{2}) \circ  \D(x_{3}) (\d(x_4))$ is illustrated in Figure~\ref{fig: ex7}. Similarly,  $V(x_3) = \D(x_{1}) \circ \D(x_{2})(\d(x_3))$ is illustrated in Figure~\ref{fig: ex8}. Finally, the vanishing cycle $V(x_2)= \D(x_{1})(\d(x_2)) = \d(x_2)$ and $V(x_1)=\d(x_1)$  by definition.

\begin{figure}[ht]  \relabelbox \small {\epsfxsize=5in
\centerline{\epsfbox{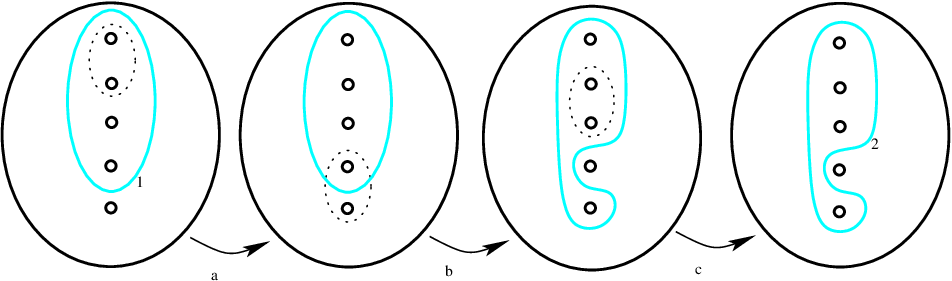}}}  \relabel{1}{$\d(x_4)$} \relabel{2}{$V(x_4)$} \relabel{a}{$\D(x_{3})$} \relabel{b}{$\D(x_{2})$} \relabel{c}{$\D(x_{1})$}
 \endrelabelbox

\caption{Starting from $\d(x_4)$, we apply a counterclockwise half-twist on the subdisk enclosed by the dotted curve, at each step, going from left to right. }\label{fig: ex7} \end{figure}

\begin{figure}[ht]  \relabelbox \small {\epsfxsize=4in
\centerline{\epsfbox{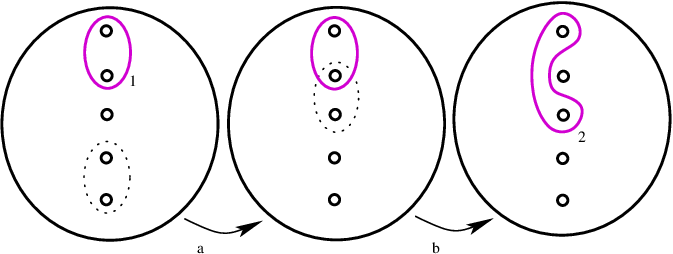}}} \relabel{1}{$\d(x_3)$} \relabel{2}{$V(x_3)$} \relabel{b}{$\D(x_{1})$} \relabel{a}{$\D(x_{2})$}
 \endrelabelbox
\caption{Starting from $\d(x_3)$, we apply a counterclockwise half-twist on the subdisk enclosed by the dotted curve, at each step, going from left to right. }\label{fig: ex8} \end{figure}

\section{Planar Lefschetz fibrations on Stein fillings of lens spaces} \label{sec: palf}

\subsection{Symplectic fillings of lens spaces} \label{sec: lens}

In \cite{l}, Lisca classified the minimal symplectic fillings of the contact $3$-manifold $(L(p,q), \xi_{can})$, up to diffeomorphism. It turns out any minimal symplectic filling of $(L(p,q), \xi_{can})$ is in fact a Stein filling. We first briefly review Lisca's classification \cite{l} of
Stein fillings of $(L(p,q), \xi_{can})$, up to diffeomorphism.

\begin{Def} \label{blowup} {\em (Blowup of a tuple of positive integers)} For any integer $r \geq 2$, a blowup of an $r$-tuple of positive
integers at the $i$th term is a map $\varphi_i\colon \mathbb{Z}^r_+\to
\mathbb{Z}^{r+1}_+$ defined by
\begin{align*}
& (n_1, \ldots, n_{i}, n_{i+1}, \ldots, n_r) \mapsto  (n_1, \ldots,
n_{i-1}, n_{i}+1,1,n_{i+1}+1, n_{i+2}, \ldots,
n_r)
\end{align*}
for any $1 \leq i \leq r-1$ and by
\begin{align*} & (n_1, \ldots,
n_r) \mapsto  (n_1, \ldots, n_{r-1},
n_r+1,1)
\end{align*}
when $i=r$. The case when  $1 \leq i \leq r-1$ is called an interior blowup, whereas the case  $i=r$ is called an exterior blowup. We also say that $(0) \to (1,1)$ is the initial blowup.
\end{Def}

Suppose that $p > q \geq 1$ are coprime integers and let
$$\frac{p}{p-q}=[b_1, b_2, \ldots, b_k]=b_1-
\cfrac{1}{b_2- \cfrac{1}{\ddots- \cfrac{1}{b_k}}}$$ be the Hirzebruch-Jung continued fraction, where $b_i \geq 2$ for $1 \leq i \leq
k$. Note that the sequence of integers $\{ b_1, b_2 \ldots, b_k\}$ is uniquely determined by the pair $(p,q)$.

For any $k \geq 2$, a $k$-tuple of positive integers $(n_1, \ldots, n_k)$ is
called {\em admissible} if each of the denominators in the continued
fraction $[n_1, \ldots, n_k]$ is positive, where we do not assume that $n_i \geq 2$. For any $k \geq 2$, let $\mathcal{Z}_{k} \subset
\mathbb{Z}^k$ denote the set of admissible $k$-tuples of positive
integers $\textbf{n}=(n_1, \ldots, n_k)$ such that $[n_1, \ldots,
n_k] =0$ and let $\mathcal{Z}_{1}=\{(0)\}$.  As a matter of fact,  any $k$-tuple of
positive integers in $\mathcal{Z}_{k}$ can be obtained from $(0)$
by a sequence of blowups as observed by Lisca \cite[Lemma 2]{l}. Note that the only possible blowup of $(0)$ is the initial blowup $(0) \to (1,1)$. Let
$$\mathcal{Z}_{k}(\textstyle{\frac{p}{p-q}}) = \{ (n_1,
\ldots, n_k)\in\mathcal{Z}_{k}\,|\, 0 \leq n_i \leq b_i \;
\mbox{for} \; i=1, \ldots , k\}.$$

Next, for every $k$-tuple $\textbf{n}=(n_1, \ldots, n_k)\in \mathcal{Z}_{k}(\textstyle{\frac{p}{p-q}})$, we describe a $4$-manifold $W_{p,q}(\textbf{n})$ whose boundary is orientation-preserving diffeomorphic to $L(p,q)$. We start with a chain of unknots in $S^3$ with framings $n_1, n_2,
\ldots, n_k$, respectively. It can be easily verified that the result of Dehn
surgery on this framed link, which we denote $N(\textbf{n})$, is diffeomorphic to $S^1 \times S^2$. Let
$\textbf{L}=\bigcup_{i=1}^{k} L_i$ denote the framed link in $N(\textbf{n})$ depicted in red
in Figure~\ref{han}, where each $L_i$ has $b_i -n_i$ components.

\begin{figure}[ht]
  \relabelbox \small {\epsfxsize=4.5in
  \centerline{\epsfbox{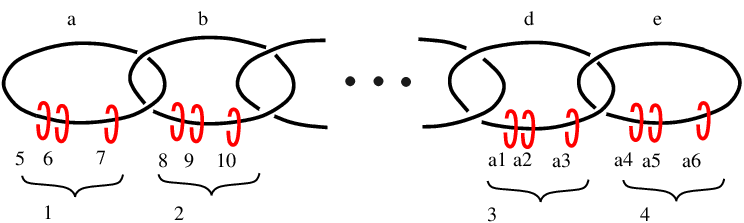}}}

\relabel{a}{$n_1$} \relabel{b}{$n_2$} \relabel{d}{$n_{k-1}$} \relabel{e}{$n_k$} \relabel{1}{$b_1-n_1$}

\relabel{2}{$b_2-n_2$} \relabel{3}{$b_{k-1}-n_{k-1}$} \relabel{4}{$b_k-n_k$} \relabel{5}{$-1$} \relabel{6}{$-1$}

\relabel{7}{$-1$} \relabel{8}{$-1$} \relabel{9}{$-1$} \relabel{10}{$-1$} \relabel{a1}{$-1$} \relabel{a2}{$-1$}

\relabel{a3}{$-1$} \relabel{a4}{$-1$} \relabel{a5}{$-1$} \relabel{a6}{$-1$}

\endrelabelbox
       \caption{The relative handlebody decomposition of the $4$-manifold $W_{(p,q)}(\textbf{n})$.}

        \label{han}
\end{figure}

Since $N(\textbf{n})$ is diffeomorphic to $S^1 \times S^2$, one can fix a diffeomorphism $\phi : N(\textbf{n})\to S^1 \times S^2$. By attaching
$2$-handles to $S^1 \times D^3$ along the framed link $\phi(\textbf{L}) \subset S^1
\times S^2$, we obtain a smooth $4$-manifold $W_{p,q}(\textbf{n})$ whose boundary is orientation-preserving diffeomorphic to $L(p,q)$.  As noted by Lisca, the diffeomorphism type of $W_{p,q}(\textbf{n})$  is independent of the choice of $\phi$ since any self-diffeomorphism
of $S^1 \times S^2$ extends to $S^1 \times D^3$.

According to Lisca,
any minimal symplectic filling (in fact Stein filling) of ($L(p,q), \xi_{can})$ is
orientation-preserving diffeomorphic to
$W_{p,q}(\textbf{n})$ for some $\textbf{n} \in
\mathcal{Z}_{k}(\frac{p}{p-q})$.

\subsection{Planar Lefschetz fibrations on Stein fillings} \label{sec: planar}

In \cite{bo}, we described an algorithm to construct a planar Lefschetz fibration $W_{p,q}(\textbf{n}) \to D^2$, based on any given blowup sequence $$(0) \to (1,1) \to \cdots \to \textbf{n}=(n_1, \ldots, n_k)\in
\mathcal{Z}_{k}(\frac{p}{p-q}).$$ Here we briefly review our algorithm, which consists of two parts, {\em stabilization} and {\em surgery}, that gives an ordered set of  vanishing cycles on a disk with $k$ holes which is the fibre of our Lefschetz fibration $W_{p,q}(\textbf{n}) \to D^2$. We begin by describing the first part of our algorithm which we call the stabilization algorithm.

\subsubsection{The stabilization algorithm} \label{sec: stab} For any positive integer $r$,  let $D_r$ denote the disk with $r$ holes. We assume that the holes are aligned horizontally in $D_r$ and we enumerate the holes in $D_r$ from left to right as $H_1, H_2, \ldots,  H_{r}.$

The initial step of the algorithm corresponding to $(0)$ is the disk $D_1$  with no vanishing cycle, as depicted on the top in Figure~\ref{fig: stabil1}. Recall that the only blowup starting from $(0)$ is the initial blowup $(0) \to (1,1)$. The corresponding fibre is the disk $D_2$ with one vanishing cycle $\a_1$, which is parallel to the boundary of $H_1$, as depicted in the middle in Figure~\ref{fig: stabil1}. This is a stabilization of the previous step, where we had the annulus $D_1$ with no vanishing cycle. Depending on the  type of the next blowup, we proceed as follows.

\begin{figure}[ht]
  \relabelbox \small {\epsfxsize=4.5in
  \centerline{\epsfbox{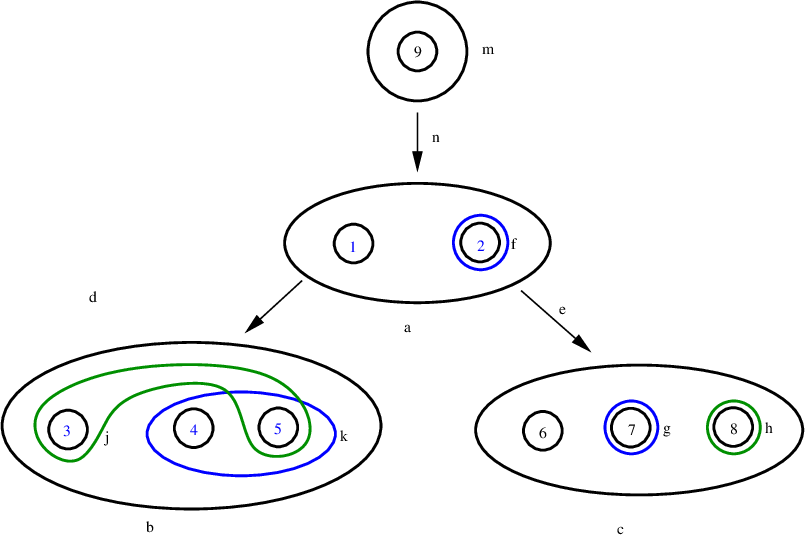}}}

\relabel{1}{$1$} \relabel{2}{$2$} \relabel{3}{$1$} \relabel{4}{$2$} \relabel{5}{$3$} \relabel{6}{$1$} \relabel{7}{$2$} \relabel{8}{$3$} \relabel{9}{$1$}

\relabel{a}{$(1,1)$} \relabel{b}{$(2,1,2)$} \relabel{c}{$(1,2,1)$} \relabel{d}{interior blowup} \relabel{e}{exterior blowup}  \relabel{n}{initial blowup}

\relabel{f}{$\a_1$} \relabel{g}{$\a_1$} \relabel{k}{$\a_1$} \relabel{h}{$\a_2$} \relabel{j}{$\a_2$} \relabel{m}{$(0)$}

\endrelabelbox
       \caption{Stabilizations depending on the type of the blowup.}
        \label{fig: stabil1}
\end{figure}

If we have an interior blowup at the first term $(1,1) \to (2,1,2)$, then $H_2$ ``splits" into two holes, where the new hole $H_3$ is placed to the right of $H_2$. The curve $\a_1$ becomes a convex curve enclosing $H_2$ and $H_3$ in $D_3$.  We introduce a new vanishing cycle $\a_2$ which encloses $H_1$ and $H_3$ in $D_3$ as shown at the bottom left in Figure~\ref{fig: stabil1}. We can view the introduction of $\a_2$ as a stabilization of the previous step.

On the other hand, if we have an  exterior blowup $(1,1) \to (1,2,1)$, then we simply introduce a new hole $H_3$ to the right, and the new vanishing cycle $\a_2$ is parallel to the boundary of $H_3$ in $D_3$ as shown at the bottom right in Figure~\ref{fig: stabil1}. Again, we can view the introduction of $\a_2$ as a stabilization of the previous step.

Now suppose that we have a set of  $r-1$ vanishing cycles $\a_1, \a_2, \ldots, \a_{r-1}$ on a disk with $r$ holes corresponding to some blowup sequence $$(0) \to (1,1) \to \cdots \to (n_1, \ldots, n_r).$$  Depending on the  type of the next blowup we insert a new hole and introduce a new vanishing cycle $\a_r$ as follows.

If we have an interior blowup at the $i$th term, for $1 \leq i \leq r-1$, then the hole $H_{i+1}$  ``splits" into two holes, where the new hole $H_{i+2}$ is placed to the right of $H_{i+1}$ in the resulting disk $D_{r+1}$. We introduce a new vanishing cycle $\a_{r}$ which encloses the holes $H_1, H_2, \ldots, H_{i}$ and the new hole $H_{i+2}$ in $D_{r+1}$. We can view the introduction of $\a_r$ as a stabilization of the previous step.

On the other hand, if we have an  exterior blowup, then we simply insert a new hole $H_{r+1}$ to the right, which is the last hole in the geometric order from the left in the resulting disk $D_{r+1}$ and the new vanishing cycle $\a_r$ is parallel to the boundary of $H_{r+1}$.  Again, we can view the introduction of $\a_r$ as a stabilization of the previous step. \smallskip

Next, we describe the second part of our algorithm which we call the surgery algorithm.

\subsubsection{The surgery algorithm} \label{sec: suralg} The surgery algorithm is based on the link  $\textbf{L}=\bigcup_{i=1}^{k} L_i$, which is used to define $W_{p,q}(\textbf{n})$. The vanishing cycles in this subsection will be mutually disjoint and hence their order does not matter. So we can describe all the vanishing cycles as a set of curves on the  disk $D_k$ with $k$ holes.

\begin{Def} \label{def: gamma} {\em (The $\g$-curves)} For each $1 \leq i \leq k$, let $\g_i$ be the convex curve on $D_k$ enclosing the holes $H_1, H_2, \ldots, H_{i}$. \end{Def}

Then the set of vanishing cycles in this part of the algorithm is $$\{ \underbrace{\g_1, \ldots,\g_1}_{b_1-n_1}, \underbrace{\g_2, \ldots, \g_2}_{b_2-n_2}, \ldots, \underbrace{\g_k, \ldots \g_k}_{b_k -n_k}\},$$ where each $\g_i$ appears $b_i -n_i$ times in the set. In particular, if $b_i=n_i$, then $\g_i$ is not in the set of vanishing cycles.

\subsubsection{Total monodromy} The fibre of the planar Lefschetz fibration $W_{p,q}(\textbf{n}) \to D^2$ is the disk $D_k$ with $k$ holes, where $k$ is  the length of the continued fraction $\frac{p}{p-q}=[b_1, b_2 \ldots, b_k]$. The set of vanishing cycles consists of the curves $\a_1, \a_2, \ldots, \a_{k-1}$ coming from the stabilization algorithm and $\g_1, \g_2, \ldots, \g_k$ (each with a multiplicity) coming from the surgery algorithm. Let $D(\a)$ denote the right-handed twist along a simple closed curve $\a$ on a surface.   The total monodromy of the planar Lefschetz fibration $W_{p,q}(\textbf{n}) \to D^2$ is given as the following composition of Dehn twists along the vanishing cycles $$D(\a_1) D(\a_2) \cdots D(\a_{k-1})  D^{b_1 - n_1} (\g_1) D^{b_2 - n_2}  (\g_2) \cdots D^{b_k - n_k}  (\g_k).$$

In Lemma~\ref{lem: turn} below, we describe another planar Lefschetz fibration on $W_{p,q}(\textbf{n})$.

\begin{Lem} \label{lem: turn} Let $f: W_{p,q}(\textbf{n}) \to D^2$ be the planar Lefschetz fibration we constructed in \cite{bo}.  The total space of the planar Lefschetz fibration obtained by reversing the order of the vanishing cycles of $f$, while taking  their mirror images   is diffeomorphic to $W_{p,q}(\textbf{n})$. \end{Lem}

\begin{proof} The result follows from the fact that such a transformation of the vanishing cycles can be achieved by rotating the absolute handlebody diagram inducing the planar Lefschetz fibration constructed in \cite{bo}.  To see this, consider for example the handlebody diagram in \cite[Figure 7]{bo}, which is depicted on the left-hand side in Figure~\ref{fig: palf}.

\begin{figure}[ht]
  \relabelbox \small {\epsfxsize=3.5in
  \centerline{\epsfbox{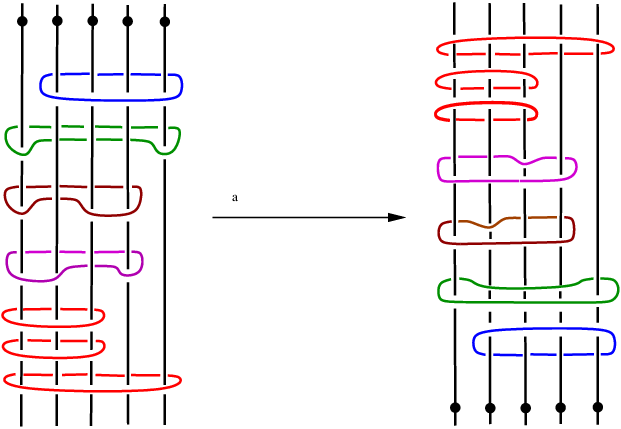}}}
  \relabel{a}{Rotate by $180^{\circ}$}
\endrelabelbox
       \caption{By rotating the handlebody diagram $180^{\circ}$ in a direction {\em normal} to the page, we obtain the mirror images of the vanishing cycles in reverse order.}

        \label{fig: palf}
\end{figure}

By rotating this handlebody diagram $180^{\circ}$ in a direction {\em normal} to the page, we get the handlebody diagram on the right-hand side whose total space is still the same. But this new handlebody diagram corresponds to a planar Lefschetz fibration, where the mirror images of the vanishing cycles appear in reverse order.   Note that here we view the base disk $D_k$ ``horizontally" and the mirror image $\overline{\a}$ of a curve $\a \subset D_k$ is defined to be the reflection of $\a$ along the $x$-axis, once the holes in $D_k$ are aligned horizontally along the $x$-axis. This definition of mirror image, of course, coincides with the mirror image in a vertical $D_k$ by rotating the horizontal $D_k$ clockwise by $90^\circ$. \end{proof}

\subsection{An example} \label{ex: example} For $p=56$ and $q=17$, we have $\dfrac{56}{56-17} = [2,2,5,2,3]$. The $5$-tuple $\textbf{n}=(2,1,4,1,2)$ belongs to $\mathcal{Z}_{5}(\frac{56}{56-17})$ since we have the blowup sequence $$(0) \to (1,1) \to (1,2,1) \to (2,1,3,1) \to (2,1,4,1,2)$$ and hence  we conclude that $W_{(56,17)}((2,1,4,1,2))$ is a Stein filling of the contact $3$-manifold $(L(56, 17), \xi_{can})$. The fibre of the planar Lefschetz fibration $$W_{(56,17)}(2,1,4,1,2) \to D^2$$ is the disk $D_5$ with $5$ holes, and to obtain the vanishing cycles $\a_1, \a_2, \a_3, \a_4$ coming from the stabilization algorithm, we start from the step $(1,2,1)$ which is already shown at the bottom right in Figure~\ref{fig: stabil1} and apply the stabilization algorithm to the interior blowups $(1,2,1) \to (2,1,3,1) \to (2,1,4,1,2)$ as depicted in Figure~\ref{fig: ex3}.

\begin{figure}[ht]
  \relabelbox \small {\epsfxsize=4.5in
  \centerline{\epsfbox{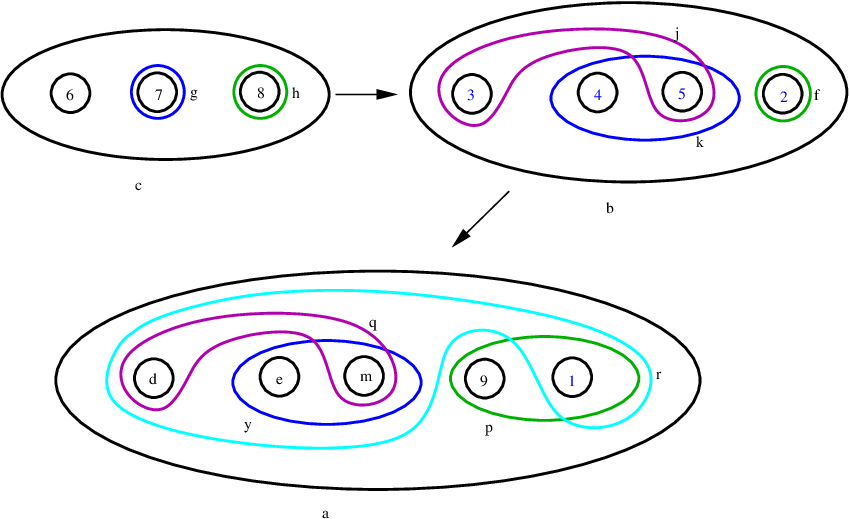}}}

\relabel{1}{$5$} \relabel{2}{$4$} \relabel{3}{$1$} \relabel{d}{$1$} \relabel{4}{$2$} \relabel{e}{$2$} \relabel{5}{$3$} \relabel{m}{$3$} \relabel{6}{$1$} \relabel{7}{$2$} \relabel{8}{$3$} \relabel{9}{$4$}

\relabel{a}{$(2,1,4,1,2)$} \relabel{b}{$(2,1,3,1)$} \relabel{c}{$(1,2,1)$}

\relabel{f}{$\a_2$} \relabel{g}{$\a_1$} \relabel{k}{$\a_1$} \relabel{y}{$\a_1$}  \relabel{p}{$\a_2$} \relabel{h}{$\a_2$} \relabel{r}{$\a_4$} \relabel{j}{$\a_3$} \relabel{q}{$\a_3$}

\endrelabelbox
       \caption{The vanishing cycles $\a_1, \a_2, \a_3, \a_4$ coming from the stabilization algorithm. }
        \label{fig: ex3}
\end{figure}

Note that $b_1 - n_1=0$, whereas $b_2 - n_2=b_3 - n_3=b_4 - n_4=b_5-n_5=1$, which implies that the set of vanishing cycles coming from the surgery algorithm in this case is $\g_2, \g_3, \g_4$ and $\g_5$ as shown in Figure~\ref{fig: ex4}.

\begin{figure}[ht]
  \relabelbox \small {\epsfxsize=3.5in
  \centerline{\epsfbox{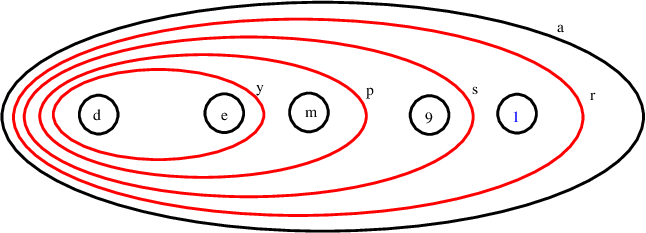}}}

\relabel{1}{$5$}  \relabel{d}{$1$}  \relabel{e}{$2$}  \relabel{m}{$3$}  \relabel{9}{$4$}

\relabel{a}{$D_5$} \relabel{y}{$\g_2$}  \relabel{p}{$\g_3$} \relabel{r}{$\g_5$} \relabel{s}{$\g_4$}

\endrelabelbox
       \caption{The vanishing cycles $\g_2, \g_3, \g_4, \g_5$ coming from the surgery algorithm. }
        \label{fig: ex4}
\end{figure}

Consequently, the total monodromy is given as the follows $$D(\a_1) D(\a_2) D(\a_3) D(\a_4)  D (\g_2) D (\g_3)  D (\g_4) D (\g_5).$$

\begin{Rem}\label{rem: monod} By Lemma~\ref{lem: turn}, there is a planar Lefschetz fibration  $W_{(56,17)}(2,1,4,1,2) \to D^2$ whose monodromy factorization is given by
$$ D (\g_5)  D (\g_4) D (\g_3) D (\g_2)  D(\overline{\a}_4) D(\overline{\a}_3) D(\overline{\a}_2) D(\overline{\a}_1).$$ \end{Rem}

\section{Unbraided wiring diagrams} \label{palf}

\subsection{The blowup algorithm} \label{sec: blowup}

In this subsection, we describe an algorithm to construct an unbraided wiring diagram corresponding to a blowup sequence starting from the  initial blowup  $(0) \to (1,1)$.   The wiring diagram corresponding to $(0)$ is a single wire $w_1$ without any marked points and the wiring diagram corresponding
to $(1,1)$, consists of two parallel wires $\{w_1, w_2\}$ so that  $w_1$ is on top without any marked points, and  $w_2$ has a single marked point $x_1$.  The next step in the algorithm depends on whether we have an interior or exterior blowup that follows  the  initial blowup  $(0) \to (1,1)$.

If we have an interior blowup at the first term $(1,1) \to (2,1,2) $ we introduce a new wire $w_3$, which is initially below $w_2$ on the right-hand side of the diagram   and as it moves to the left,  it goes through the marked point $x_1$ on $w_2$, but otherwise remains parallel to $w_2$ and then intersects $w_1$ at a new marked point $x_2$, which is to the left of $x_1$. This diagram with three wires $\{w_1, w_2, w_3\}$  corresponds to $(2,1,2)$,  which we depicted in Figure~\ref{fig: initial1}.

\begin{figure}[ht]  \relabelbox \small {\epsfxsize=4.5in
\centerline{\epsfbox{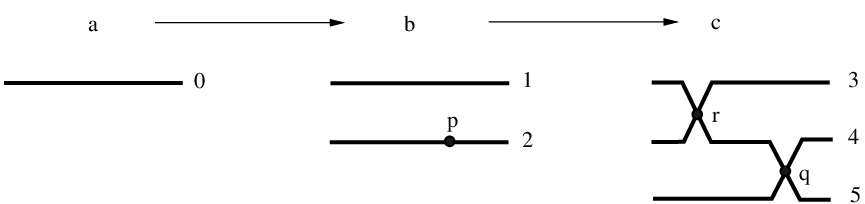}}}
\relabel{a}{$(0)$} \relabel{b}{$(1,1)$} \relabel{c}{$(2,1,2)$} \relabel{p}{$x_1$} \relabel{r}{$x_2$}  \relabel{q}{$x_1$} \relabel{0}{$w_1$}  \relabel{3}{$w_1$} \relabel{1}{$w_1$} \relabel{4}{$w_2$} \relabel{2}{$w_2$}  \relabel{5}{$w_3$} \endrelabelbox
\caption{Wiring diagrams corresponding to the blowup sequence $(0) \to(1,1) \to (2,1,2)$.} \label{fig: initial1} \end{figure}

On the other hand, if we have an  exterior blowup $(1,1) \to (1,2,1)$,  we insert in the diagram a new wire $w_3$ which is right below $w_2$ and parallel to it. We place a marked point  $x_2$ on $w_3$  so that  $x_2$ is to the left of $x_1$. This diagram with three wires $\{w_1, w_2, w_3\}$  corresponds to $(1,2,1)$, which we depicted in Figure~\ref{fig: initial2}.

\begin{figure}[ht]
  \relabelbox \small {\epsfxsize=4.5in
  \centerline{\epsfbox{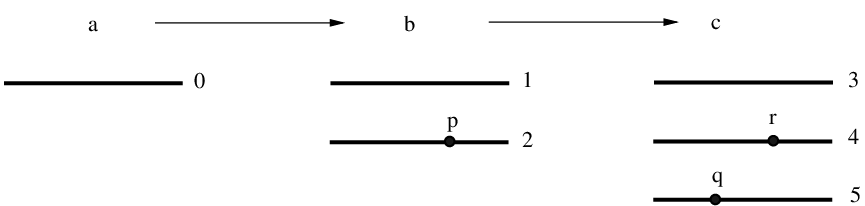}}}
\relabel{a}{$(0)$} \relabel{b}{$(1,1)$} \relabel{c}{$(1,2,1)$} \relabel{p}{$x_1$} \relabel{r}{$x_1$}  \relabel{q}{$x_2$} \relabel{0}{$w_1$}  \relabel{3}{$w_1$} \relabel{1}{$w_1$}  \relabel{4}{$w_2$} \relabel{2}{$w_2$}  \relabel{5}{$w_3$}
\endrelabelbox \caption{Wiring diagrams corresponding to the blowup sequence $(0) \to(1,1) \to (1,2,1)$.} \label{fig: initial2} \end{figure}

Now suppose that we have an unbraided wiring diagram $\mathcal{W}$ consisting of $r$ wires $\{w_1, w_2$, $\ldots, w_{r}\}$ corresponding to some blowup sequence starting from the  initial blowup  $(0) \to (1,1)$ and ending with some $r$-tuple of positive integers.  We would like to emphasize that the indices of the wires in the set $\mathcal{W}$ above indicate the order in which the wires are introduced into the diagram. Depending on the  type of the next blowup we insert a new wire in $\mathcal{W}$ and adjust the diagram accordingly as follows.

Suppose that we have an interior blowup at the $i$th term, for some $1 \leq i \leq r-1$. Let $w_j \in \mathcal{W}$ be the $(i+1)$st wire with respect to the {\em geometric ordering} of the wires on the right-hand side of the diagram, and let $\mathcal{W}_i$ denote the subset of $\mathcal{W}$  consisting of all the wires which appears before $w_j$ in this ordering. In other words, $\mathcal{W}_i$ is the set of the top $i$ wires in the geometric ordering of the wires on the right-hand side of the diagram.  Now we introduce a new wire, named $w_{r+1}$, into the diagram, which is initially right below $w_j$ on the right-hand side of the diagram  and as it moves to the left, goes through all the marked points on $w_j$ but otherwise remains parallel to $w_j$, and then  we insert a new marked point $x_r$ on $w_{r+1}$ which is the simultaneous intersection of  $w_{r+1}$ and  all the wires in $\mathcal{W}_i$.    We place the marked point $x_r$ to the left of $x_{r-1}$. For this to work, we need to know that  the set  $\mathcal{W}_i  \cup \{w_{r+1}\}$ of wires is geometrically consecutive on the left-hand side, which we verify in Lemma~\ref{lem: consec} below, where we refer to this step in the algorithm as the {\em last twist}.

On the other hand, if we have an  exterior blowup, we insert a new wire $w_{r+1}$ below all the wires in $\mathcal{W}$ with no intersection points with the other wires, and place a single  marked point $x_r$ on $w_{r+1}$, which is to the left of $x_{r-1}$.

We  call this procedure the blowup algorithm for wiring diagrams. Note that in the resulting wiring diagram, the wires are indexed in the order they are introduced into the diagram but their geometric ordering on the right-hand side (or the left-hand side) of the diagram as viewed on the page, might be different from the index ordering. Moreover,  by our algorithm, $w_1$ will  always be at the top on the right-hand side of the diagram.

\begin{Lem} \label{lem: consec} If $\mathcal{W}$ is an unbraided wiring diagram consisting of wires $\{w_1, w_2, \ldots, w_{r}\}$, which is obtained by the blowup algorithm with respect to some blowup sequence starting from the  initial blowup  $(0) \to (1,1)$, then any set of wires including $w_1$, which is consecutive with respect to the geometric ordering on the right-hand side of the diagram,  is also geometrically consecutive (perhaps with a different geometric ordering) on the left-hand side of the diagram. Moreover, if any wire other than $w_1$ carries an even (resp. odd) number of marked points, then on the left-hand side it is above (resp. below) all the wires which appear before it in the geometric ordering of the wires on the right-hand side of the diagram.
\end{Lem}

\begin{proof} We prove the lemma by induction on the number of wires. The two wiring diagrams we described above corresponding to the blowup sequences $(0) \to (1,1) \to (2,1,2) $ and $(0) \to (1,1) \to (1,2,1)$, respectively, can be taken to be the initial step of our induction argument. The properties stated in Lemma~\ref{lem: consec}  hold for these wiring diagrams.

Suppose that both properties stated in Lemma~\ref{lem: consec} hold when there are up to $r \geq 3$ wires in any unbraided wiring diagram obtained as a result of the blowup algorithm with respect to some blowup sequence starting from the  initial blowup  $(0) \to (1,1)$.    We will prove that these properties continue to hold when a new wire is inserted into the diagram corresponding to a new blowup.  If the new wire inserted corresponds to an exterior blowup, it is clear that both properties stated in Lemma~\ref{lem: consec} continue to hold in the  new diagram with $r+1$ wires. This is because in this case, the new wire will be inserted at the bottom of the diagram with a single marked point on it and without any intersections with the other wires.

Suppose that a new wire $w_{r+1}$ is inserted into $\mathcal{W}$ with respect to an interior blowup at the $i$th term, for some $1 \leq i  \leq r-1$.   Let $\mathcal{W}_{s}$ be the subset  of $\mathcal{W}$ consisting of the top $s$ wires in the {\em geometric ordering} of the wires on the right-hand side of the diagram. Note that $\mathcal{W}_{i+1} = \mathcal{W}_i \cup \{w_j\}$, since by definition, $w_j \in \mathcal{W}$ is the $(i+1)$st wire with respect to the geometric ordering of the wires on the right-hand side of the diagram.

Assume that $w_j$ has an odd number of marked points. By the induction hypotheses, before we insert $w_{r+1}$, the wires in the set  $\mathcal{W}_{i+1}$ are geometrically consecutive (perhaps with a different geometric ordering) on  the left-hand side,  while  $w_j$ is at the bottom of these geometrically  consecutive wires. The new wire $w_{r+1}$ will be initially right below the wire $w_j$ on the right-hand side of the diagram and $w_{r+1}$ will go through all the marked points on $w_j$, and otherwise it will remain parallel to $w_j$, before the last twist in the algorithm.   But since $w_j$ has an odd number of marked points, and $w_{r+1}$ is initially right below $w_j$, the wire $w_{r+1}$ will be right above $w_j$ on the left-hand side before the last twist. Therefore, before the last twist, the wires in the set $\mathcal{W}_{i+1} \cup \{w_{r+1}\}$ will be geometrically consecutive  on the left-hand side, and moreover $w_{r+1}, w_j$ will be the bottom two wires in that order.  Finally, when we twist once all the wires in the set $\mathcal{W}_i \cup \{w_{r+1}\} $ (to create a simultaneous intersection point of these $i+1$ wires) as part of the blowup algorithm, the wires in the set  $\mathcal{W}_{i+1} \cup \{w_{r+1}\}$ will remain geometrically consecutive  on the left-hand side, where $w_{r+1}$ will appear at the top, and $ w_j$ will appear at the bottom of this consecutive set of wires.

Assume that $w_j$ has an even number of marked points. By the induction hypotheses, before we insert $w_{r+1}$, the wires in the set  $\mathcal{W}_{i+1}$ are geometrically consecutive (perhaps with a different geometric ordering) on  the left-hand side,  while  $w_j$ is at the top of these geometrically  consecutive wires. The new wire $w_{r+1}$ will be initially right below the wire $w_j$ on the right-hand side of the diagram and $w_{r+1}$ will go through all the marked points on $w_j$, and otherwise it will remain parallel to $w_j$, before the last twist in the algorithm.   But since $w_j$ has an even number of marked points, and $w_{r+1}$ is initially right below $w_j$, the wire $w_{r+1}$ will be right below $w_j$ on the left-hand side before the last twist. Therefore, before the last twist, the wires in the set $\mathcal{W}_{i+1} \cup \{w_{r+1}\}$ will be geometrically consecutive  on the left-hand side and, moreover, $w_j, w_{r+1}$ will be the top two wires in that order.  Finally, when we twist once all the wires in the set $\mathcal{W}_i \cup \{w_{r+1}\} $ (to create a simultaneous intersection point of these $i+1$ wires) as part of the blowup algorithm, the wires in the set  $\mathcal{W}_{i+1} \cup \{w_{r+1}\}$ will remain geometrically consecutive  on the left-hand side, where $w_j$ will appear at the top, and $w_{r+1}$ will appear at the bottom of this consecutive set of wires.

The discussion above proves that, after we insert $w_{r+1}$, any set of wires in $\mathcal{W} \cup \{w_{r+1}\}$ including $w_1$, which is consecutive with respect to the geometric ordering on the right-hand side of the diagram,  is also geometrically consecutive (perhaps with a different geometric ordering) on the left-hand side of the diagram.

Moreover, if $w_j$ has an odd (resp. even) number of marked points, then $w_{r+1}$ will have even (resp.odd) number of marked points by the blowup algorithm and it will be above (resp. below) all the wires in $\mathcal{W}_{i+1}$ on the left-hand side of the diagram.  The upshot is that both properties stated in Lemma~\ref{lem: consec}  hold true for the unbraided wiring diagram $\mathcal{W} \cup \{w_{r+1}\}$.
\end{proof}

\subsection{An example} Consider the blowup sequence $$(0) \to (1,1) \to (1,2,1) \to (2,1,3,1) \to (2,1,4,1,2).$$ In Figure~\ref{fig: example} below we depict the diagrams corresponding to  $$(1,2,1) \to (2,1,3,1) \to (2,1,4,1,2)$$ starting from the diagram of $(1,2,1)$ already depicted in  Figure~\ref{fig: initial2}.

\begin{figure}[ht]  \relabelbox \small {\epsfxsize=4.5in
\centerline{\epsfbox{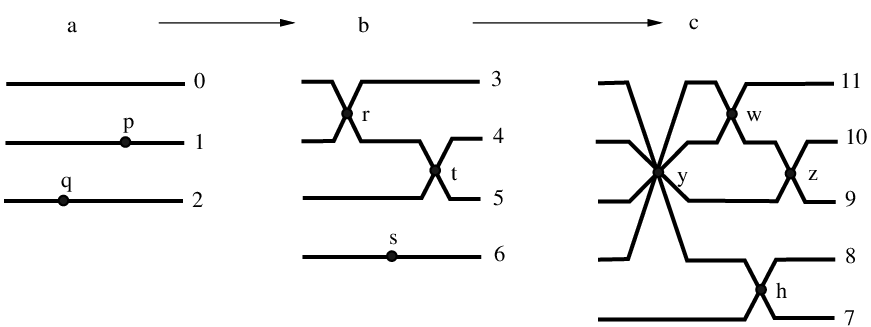}}}
\relabel{a}{$(1,2,1)$} \relabel{b}{$(2,1,3,1)$} \relabel{c}{$(2,1,4,1,2)$} \relabel{p}{$x_1$} \relabel{z}{$x_1$} \relabel{t}{$x_1$} \relabel{r}{$x_3$}  \relabel{w}{$x_3$} \relabel{y}{$x_4$} \relabel{q}{$x_2$} \relabel{h}{$x_2$}  \relabel{s}{$x_2$} \relabel{0}{$w_1$} \relabel{3}{$w_1$} \relabel{11}{$w_1$}  \relabel{2}{$w_3$} \relabel{6}{$w_3$} \relabel{8}{$w_3$} \relabel{1}{$w_2$}  \relabel{4}{$w_2$}  \relabel{10}{$w_2$}   \relabel{5}{$w_4$} \relabel{9}{$w_4$} \relabel{7}{$w_5$} \endrelabelbox
\caption{Wiring diagrams corresponding to the blowup sequence $(1,2,1) \to (2,1,3,1) \to (2,1,4,1,2)$.}\label{fig: example} \end{figure}

\subsection{The twisting algorithm} \label{sec: twisting}

Suppose that $\mathcal{W}$ is an unbraided wiring diagram consisting of wires $\{w_1, w_2, \ldots, w_{k}\}$, which is obtained by the blowup algorithm with respect to some blowup sequence, starting from the  initial blowup  $(0) \to (1,1)$ and ending with some $k$-tuple of positive integers.   Let $\mathcal{W}_{s}$ be the subset  of $\mathcal{W}$ consisting of the top $s$ wires in the {\em geometric ordering} of the wires on the right-hand side of the diagram,  as in Section~\ref{sec: blowup}.  Based on any $k$-tuple $\textbf{m}=(m_1,\ldots, m_{k})$, where $m_i$ is a nonnegative integer,  we describe a procedure called the {\em twisting algorithm} to extend the unbraided wiring diagram $\mathcal{W}$ to another unbraided wiring diagram $\mathcal{W} ({\textbf{m})}$ with the same number of wires but with more marked points obtained by extra twists inserted to the left.

If $m_1=0$, then we do not modify $\mathcal{W}$ and move onto the next step. If $m_1> 0$, then we simply add $m_1$ extra marked points $\underbrace{y_1, y_1,  \ldots, y_1}_{m_1}$ on $w_1$ to the left of $x_{k-1}$. If $m_2=0$, then we do not modify the diagram any further and move onto the next step. If $m_2 > 0$, then by Lemma~\ref{lem: consec}, we know that the wires in $\mathcal{W}_2$ are geometrically consecutive on the left-hand side of the diagram $\mathcal{W}$. We extend $\mathcal{W}$ by  twisting $m_2$-times the wires in $\mathcal{W}_2$, creating consecutive simultaneous intersection points $\underbrace{y_{2}, y_{2},  \ldots, y_{2}}_{m_2}$ to the left of the last, if any, $y_{1}$.  If $m_3=0$, then we do not modify the diagram any further and move onto the next step. Now suppose that $m_3 > 0$. Since the wires in $\mathcal{W}_3$ are geometrically consecutive on the left-hand side of the diagram $\mathcal{W}$ by Lemma~\ref{lem: consec}  these wires will remain geometrically consecutive after the first additional twists we possibly put into the diagram corresponding to $m_2$. We extend the diagram further by  twisting $m_3$-times the wires in $\mathcal{W}_3$, creating simultaneous intersection point $\underbrace{y_{3}, y_{3},  \ldots, y_{3}}_{m_3}$ to  the left of the last, if any,  $y_2$.  By iterating this procedure, we extend $\mathcal{W}$ to  $\mathcal{W} ({\textbf{m})}$ with additional marked points corresponding to \textbf{m}.

\begin{Rem} Here, we think of $m_i$ as the ``multiplicity" of the point $y_i$. If $m_i=0$, then $y_i$ does not appear in the diagram, and if $m_i > 1$, then $y_i$ is repeated $m_i$-times. To avoid cumbersome  notation, we do not put an extra index to distinguish between different $y_i$ type points.  \end{Rem}

\subsection{An example} Here we give an example where we extend the wiring diagram $\mathcal{W}$ corresponding to the blowup sequence   $$(0) \to (1,1) \to (1,2,1) \to (2,1,3,1) \to (2,1,4,1,2)$$ depicted in Figure~\ref{fig: example} to $\mathcal{W}({\textbf{m}})$  applying the twisting algorithm based on  $\textbf{m}=(0,1,1,1,1)$.  Note that inside the dotted square in Figure~\ref{fig: example2}, there is a copy of $\mathcal{W}$ from Figure~\ref{fig: example}.

\begin{figure}[ht]  \relabelbox \small {\epsfxsize=4.5in
\centerline{\epsfbox{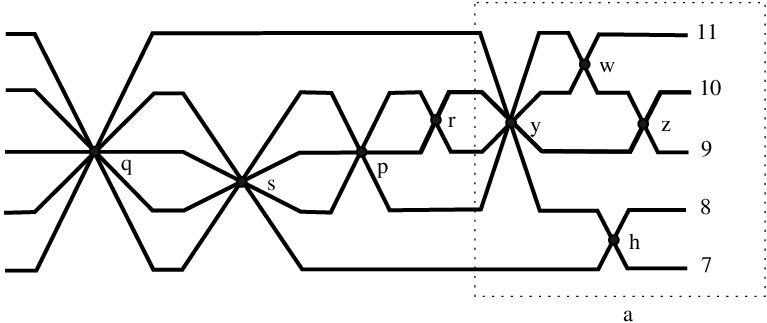}}}
 \relabel{a}{$\mathcal{W}$} \relabel{p}{$y_{3}$} \relabel{z}{$x_1$}  \relabel{r}{$y_{2}$}  \relabel{w}{$x_3$} \relabel{y}{$x_4$} \relabel{q}{$y_{5}$} \relabel{s}{$y_{4}$} \relabel{h}{$x_2$}  \relabel{11}{$w_1$}   \relabel{8}{$w_3$}   \relabel{10}{$w_2$}  \relabel{9}{$w_4$} \relabel{7}{$w_5$} \endrelabelbox
\caption{Extending $\mathcal{W}$ to $\mathcal{W} ((0,1,1,1,1))$ by applying the surgery algorithm based on $\textbf{m}=(0,1,1,1,1)$.}\label{fig: example2} \end{figure}

In this example $m_1=0$ and $m_2=1$ and the wires $\mathcal{W}_2=\{w_1, w_2\}$ are geometrically consecutive on the left-hand side of $\mathcal{W}$. Now we twist them together once to obtain the marked point $y_{2}$, which is to the left of $x_4$. Since $m_3=1$, next we twist the wires in $\mathcal{W}_3=\{w_1, w_2,w_4\}$ (which are geometrically consecutive) together once to obtain the marked point $y_{3}$, which is to the left of $y_{2}$.  Since $m_4=1$, we twist the wires in $\mathcal{W}_4=\{w_1, w_2, w_4, w_3\}$ (which are geometrically consecutive) together once to obtain the marked point $y_{4}$, which is to the left of $y_{3}$. Finally, since $m_5=1$, we twist all the wires in $\mathcal{W}_5 = \mathcal{W}$ together once to obtain the marked point $y_{5}$, which is to the left of $y_{4}$, as illustrated in Figure~\ref{fig: example2}.

\begin{Rem} \label{rem: defns} We will also speak about $\d(y_s)$, $\D(y_s)$ and $V(y_s)$ for each marked point $y_s$ in the rest of the paper, as described in Definitions~\ref{def: convex}, ~\ref{def: twist}, and ~\ref{def: cycle}. \end{Rem}

\section{From vanishing cycles to unbraided wiring diagrams}

We recall the main theorem from the introduction, where we have replaced $W$ with $W_{p,q}(\textbf{n})$ below, to be more precise. \smallskip

\noindent {\bf Theorem~\ref{thm: main}.} {\em There is an algorithm to draw an explicit unbraided wiring diagram whose associated planar Lefschetz fibration obtained by the method of Plamenevskaya and  Starkston \cite{ps} is equivalent to the planar Lefschetz fibration $W_{p,q}(\textbf{n}) \to D^2$  constructed by the authors in \cite{bo}.} \smallskip

Before we give the proof of Theorem~\ref{thm: main} below, we illustrate the statement and its proof on an example. First we introduce some notation that will be used  in the following discussion.  The  disk $D_k$ with $k$ holes will be viewed in two different but equivalent ways as follows: (i) the holes are aligned horizontally in $D_k$ and enumerated from left to right  or (ii) the holes are aligned vertically in $D_k$ and enumerated from top to bottom. Here we identify the ``horizontal"  $D_k$ in (i) with the ``vertical" $D_k$ in (ii) by rotating the ``horizontal"  $D_k$  clockwise by $90^\circ$. The reason why we consider these two embeddings of a disk with holes is that the vanishing cycles in \cite{bo} are described on a horizontal  $D_k$, while the vanishing cycles in \cite{ps} are described on a vertical $D_k$. Here we compare them on a  vertical $D_k$ via the identification given above. When we view $D_k$ vertically, the mirror image $\overline{\a}$ of a curve $\a \subset D_k$ is defined to be the reflection of $\a$ along the $y$-axis, once the holes in $D_k$ are aligned vertically along the $y$-axis.

\subsection{An example} \label{subsec: exa} In Section~\ref{ex: example}, we constructed a planar Lefchetz fibration $$W_{(56, 17)}((2,1,4,1,2)) \to D^2$$ whose fibre is the disk $D_5$ with $5$ holes and whose vanishing cycles are the curves $$\a_1, \a_2, \a_3, \a_4, \g_2, \g_3, \g_4, \g_5$$ in $D_5$, which are depicted in Figures~\ref{fig: ex3} and ~\ref{fig: ex4}. We claim that the planar Lefschetz fibration obtained by using the method of Plamenevskaya and  Starkston associated to the unbraided wiring diagram $W ((0,1,1,1,1))$ in Figure~\ref{fig: example2} has exactly the same set of vanishing cycles (viewed in a vertical $D_5$), except that we have to take ``mirror images" of all the curves and {\em reverse the order}  of the vanishing cycles in the total monodromy.  In other words, first we rotate the disks in Figures~\ref{fig: ex3} and ~\ref{fig: ex4} clockwise by $90^\circ$ and then take the mirror images of the curves. This modification of the vanishing cycles is not an issue by Lemma~\ref{lem: turn}. Note that the mirror image of a $\g$-curve is equal to itself, and hence we only need to take the mirror images of  the $\a$-curves.  As a matter of fact, we claim  that  $V(y_{j})=\g_j$, for $ 2 \leq j \leq 5$   and $V(x_i) = \overline{\a}_i$, for $ 1 \leq i \leq 4$ (see Remark~\ref{rem: defns} for notation). To verify our claim, we apply the method of Plamenevskaya and  Starkston (see Section~\ref{sec: bwd}), to describe a set of ordered vanishing cycles associated to the marked points in Figure~\ref{fig: example2}, where we depicted the convex curves assigned to the marked points in Figure~\ref{fig: ex9}.

\begin{figure}[ht]  \relabelbox \small {\epsfxsize=5.5in
\centerline{\epsfbox{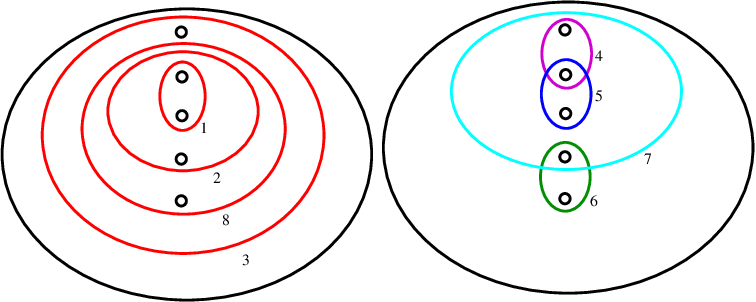}}}
 \relabel{1}{$\d(y_2)$}  \relabel{2}{$\d(y_3)$}  \relabel{3}{$\d(y_5)$} \relabel{4}{$\d(x_3)$}  \relabel{5}{$\d(x_1)$}  \relabel{6}{$\d(x_2)$} \relabel{7}{$\d(x_4)$} \relabel{8}{$\d(y_4)$}
 \endrelabelbox
\caption{Convex curves in $D_5$ assigned to the marked points in  Figure~\ref{fig: example2}.} \label{fig: ex9} \end{figure}

Note that $$V(y_{5})= \D(x_1) \circ \cdots \circ \D(x_4)\circ \D(y_2) \circ \D(y_3) \circ \D(y_4) (\d(y_5)) = \d(y_5) = \g_5,$$ $$ V(y_{4})= \D(x_1) \circ \cdots \circ \D(x_4)\circ \D(y_2) \circ \D(y_3)  (\d(y_4)) = \g_4, \; \mbox{as illustrated in Figure~\ref{fig: ex5c}},$$  $$ V(y_{3})= \D(x_1) \circ \cdots \circ \D(x_4)\circ \D(y_2)  (\d(y_3)) = \g_3, \; \mbox{as illustrated in Figure~\ref{fig: ex5}, and}$$ $$ V(y_{2})=  \D(x_1) \circ \cdots \circ \D(x_4) (\d(y_2)) = \g_2, \; \mbox{as illustrated in Figure~\ref{fig: ex6}}. $$

\begin{figure}[ht]  \relabelbox \small {\epsfxsize=5in
\centerline{\epsfbox{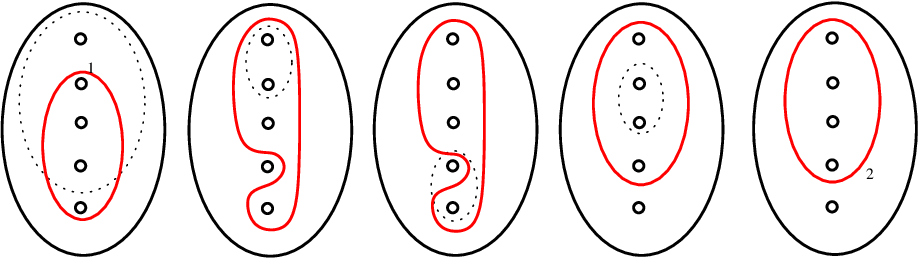}}}
\relabel{1}{$\d(y_4)$}  \relabel{2}{$\g_4$}
 \endrelabelbox
\caption{Starting from $\d(y_4)$, we apply a counterclockwise half-twist on the subdisk enclosed by the dotted curve, at each step from left to right. }\label{fig: ex5c} \end{figure}

\begin{figure}[ht]  \relabelbox \small {\epsfxsize=5in
\centerline{\epsfbox{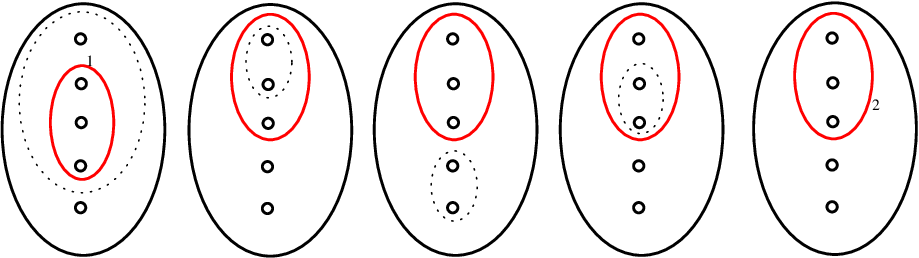}}}
\relabel{1}{$\d(y_3)$}  \relabel{2}{$\g_3$}
 \endrelabelbox
\caption{Starting from $\d(y_3)$, we apply a counterclockwise half-twist on the subdisk enclosed by the dotted curve, at each step from left to right. }\label{fig: ex5} \end{figure}

\begin{figure}[ht]  \relabelbox \small {\epsfxsize=5in
\centerline{\epsfbox{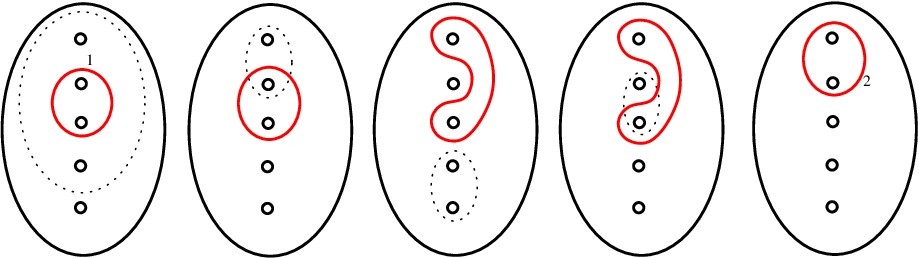}}}
\relabel{1}{$\d(y_2)$}  \relabel{2}{$\g_2$}
 \endrelabelbox
\caption{Starting from $\d(y_2)$, we apply a counterclockwise half-twist on the subdisk enclosed by the dotted curve, at each step from left to right.  }\label{fig: ex6} \end{figure}

\begin{Rem} In Figure~\ref{fig: ex5c}, we have not included $\d(y_2)$ and  $\d(y_3)$ as dotted curves  since  $\D(y_2) \circ \D(y_3)$ would not have any effect on $\d(y_4)$. Similarly, we have not included $\d(y_2)$ as a dotted curve in Figure~\ref{fig: ex5} since $\D(y_2)$ would not have any effect on $\d(y_3)$. We will generalize this observation as Lemma~\ref{lem: typex} in Section~\ref{subsec: gamma}.
\end{Rem}

Moreover, $V(x_4)= \overline{\a}_4$ by comparing Figure~\ref{fig: ex3} and Figure~\ref{fig: ex7}; $V(x_3)= \overline{\a}_3$ by comparing Figure~\ref{fig: ex3} and Figure~\ref{fig: ex8}, and finally $V(x_2)= \d(x_2) = \overline{\a}_2 = \a_2$ and  $V(x_1)= \d(x_1)= \overline{\a}_1 =\a_1$,  by comparing Figure~\ref{fig: ex3} and Figure~\ref{fig: ex9b}. Note that the total monodromy of the planar Lefschetz fibration is $$ D (\g_5)  D (\g_4) D (\g_3) D (\g_2)  D(\overline{\a}_4) D(\overline{\a}_3) D(\overline{\a}_2) D(\overline{\a}_1),$$ which coincides with the monodromy in Remark~\ref{rem: monod}.  \smallskip

Now we are ready to give a proof of Theorem~\ref{thm: main}.

\subsection{Proof of the main result} \label{subsec: pr} Suppose that $p > q \geq 1$ are coprime integers and let
$$\dfrac{p}{p-q}=[b_1, b_2, \ldots, b_k]$$ be the Hirzebruch-Jung continued fraction, where $b_i \geq 2$ for $1 \leq i \leq k$. We set $$\textbf{b}= (b_1, b_2, \ldots, b_k).$$

\begin{Def} \label{def: wd} {\em (The wiring diagrams $\mathcal{W}_\textbf{n}$, and $\mathcal{W}_\textbf{n}({\textbf{m}})$)} For any $$\textbf{n}= (n_1, n_2, \ldots, n_k) \in \mathcal{Z}_{k}(\textstyle{\frac{p}{p-q}})$$ let $(0) \to (1,1) \to \cdots \to \textbf{n}$ be a blowup sequence, and let $\textbf{m}= \textbf{b}-\textbf{n}$.  We denote by $\mathcal{W}_\textbf{n}$, the unbraided wiring diagram with $k$ wires $\{w_1, w_2, \ldots, w_{k}\}$ and $k-1$ marked points $x_{k-1}, x_{k-2}, \ldots, x_1$ (reading from left to right) constructed by applying the blowup algorithm in Section~\ref{sec: blowup} to the given blowup sequence. We denote by $\mathcal{W}_\textbf{n}({\textbf{m}})$ the extension of $\mathcal{W}_\textbf{n}$ to the left obtained by applying the twisting algorithm in Section~\ref{sec: twisting} based on the $k$-tuple $\textbf{m}$. Note that $\mathcal{W}_\textbf{n}({\textbf{m}})$ is obtained from $\mathcal{W}_\textbf{n}$ by inserting additional  marked points $$ \underbrace{y_{k},  \ldots,  y_{k}}_{m_k}, \underbrace{y_{k-1}, \ldots,  y_{k-1}}_{m_{k-1}}, \dots,  \underbrace{y_{1}, \ldots, y_{1}}_{m_1}, $$ reading from left to right. \end{Def}

\smallskip

\noindent {\bf Vanishing cycles associated to $\mathcal{W}_\textbf{n}({\textbf{m}})$:} Now we can apply the method of Plamenevskaya and  Starkston (see Section~\ref{sec: bwd}) to the wiring diagram $\mathcal{W}_\textbf{n}({\textbf{m}})$, to obtain the associated planar Lefschetz fibration by describing a set of ordered vanishing cycles on the disk $D_k$ with $k$ holes. According to their algorithm, there is a vanishing cycle associated to each marked point in  $\mathcal{W}_\textbf{n}({\textbf{m}})$. So,  for each $1 \leq t \leq  k-1$, there is  a vanishing cycle $V(x_t)$ associated to the marked point  $x_t$ in $\mathcal{W}_\textbf{n}({\textbf{m}})$, and   for each $1 \leq s \leq  k$, there is  a vanishing cycle $V(y_s)$ associated to the marked point  $y_s$ in $\mathcal{W}_\textbf{n}({\textbf{m}})$. Note that there are $k-1$ vanishing cycles associated to type $x$ marked points, and since each $y_t$ is repeated $m_t = b_t - n_t$ times, there are $$m_1+m_2+\cdots+m_k = (b_1-n_1)+(b_2-n_2)+ \cdots +(b_k-n_k)$$ vanishing cycles  in total associated to type $y$  marked points. \smallskip

\noindent {\bf Planar Lefschetz fibration $W_{p,q}(\textbf{n}) \to D^2$:}  Let  $W_{p,q}(\textbf{n})$ be the minimal symplectic filling of  $(L(p,q), \xi_{can})$ as in Section~\ref{sec: lens}. As we described in Section~\ref{sec: planar}, there is a planar Lefschetz fibration $W_{p,q}(\textbf{n}) \to D^2$ with fibre $D_k$,  which is obtained by applying the stabilization algorithm and the surgery algorithm. Note that there are $k-1$ vanishing cycles $\a_1, \a_2, \ldots, \a_{k-1}$ coming from the  stabilization algorithm and $$(b_1-n_1)+(b_2-n_2)+ \cdots +(b_k-n_k)$$ vanishing cycles  $$\{ \underbrace{\g_1, \ldots,\g_1}_{b_1-n_1}, \underbrace{\g_2, \ldots, \g_2}_{b_2-n_2}, \ldots, \underbrace{\g_k, \ldots \g_k}_{b_k -n_k}\}$$ coming from the  surgery algorithm. \smallskip

Theorem~\ref{thm: main} is in fact equivalent to  Proposition~\ref{prop: mirror} coupled with Lemma~\ref{lem: turn}.

\begin{Prop} \label{prop: mirror} Let $\mathcal{W}_\textbf{n}({\textbf{m}})$ be an unbraided wiring diagram with $k$ wires as described in Definition~\ref{def: wd}. Then
\begin{enumerate}[\rm(a)]
\item for any $1 \leq s \leq k$, the vanishing cycle $V(y_s)$ associated to the marked point $y_s \in \mathcal{W}_\textbf{n}({\textbf{m}})$  is isotopic to $\g_s$ in $D_k$, and
\item for any $1 \leq t \leq k-1$, the vanishing cycle $V(x_t)$ associated to the marked point $x_{t} \in \mathcal{W}_\textbf{n}({\textbf{m}})$ is isotopic to $\overline{\a}_{t}$ (the mirror image of $\a_t$) in $D_k$.
\end{enumerate}\end{Prop}

\smallskip

In the rest of Section~\ref{subsec: pr}, we will provide a proof of Proposition~\ref{prop: mirror}. In Section~\ref{subsec: gamma}, we will first formulate Proposition~\ref{prop: lobe} (a necessarily very technical result) and Lemma~\ref{lem: lobetomirror} will show that it implies Proposition~\ref{prop: mirror}(a). Then we will turn our attention to Proposition~\ref{prop: mirror}(b) in Section~\ref{subsec: alpha}, where we will formulate the result as Proposition~\ref{prop: alpha}.

\subsubsection{The case of $\g$-curves:}  \label{subsec: gamma} To prove our claim in Proposition~\ref{prop: mirror} (a), we will verify that for $1 \leq s \leq k$, the vanishing cycle $V(y_{s})$ is isotopic to the curve $\g_s$ in $D_k$. We begin with a simple but crucial observation.

\begin{Lem} \label{lem: typex} For any wiring diagram $\mathcal{W}_\textbf{n}({\textbf{m}})$ with $k$ wires as in Definition~\ref{def: wd}, and for any $1 \leq s \leq k$, we have $$V(y_s) = \D(x_1) \circ \cdots \circ \D(x_{k-1}) (\d (y_s)). $$
 \end{Lem}

\begin{proof} By Definition~\ref{def: cycle} we have
$$V(y_s) = \D(x_1) \circ \cdots \circ \D(x_{k-1}) \circ (\D(y_{1})) ^{b_1-n_1} \circ \cdots \circ  (\D(y_{s-1})) ^{b_{s-1}-n_{s-1}} (\d (y_s)). $$  But $$ (\D(y_{1})) ^{b_1-n_1} \circ \cdots \circ  (\D(y_{s-1})) ^{b_{s-1}-n_{s-1}} (\d (y_s)) =  \d (y_s),$$ since the convex curves associated to the type $y$ marked points are nested, due to the construction and the order of the type $y$ marked points in the wiring diagram. \end{proof}

Therefore, to prove our claim in Proposition~\ref{prop: mirror} (a), for each $1 \leq s \leq k$, we need to verify that $$  \D(x_{1}) \circ \cdots \circ \D(x_{k-1}) (\d(y_s))   =  \g_s$$ by Definition~\ref{def: cycle} and Lemma~\ref{lem: typex}. Equivalently, we need to verify that for each $1 \leq s \leq k$,
$$(\D(x_{k-1}))^{-1} \circ \cdots \circ (\D(x_{1}))^{-1} ( \g_s) = \d(y_s).$$  For technical reasons, we will prove a more refined
statement in Proposition~\ref{prop: lobe} from which our claim  will follow by Lemma~\ref{lem: lobetomirror}. Before giving the statement we make the following definition.

\begin{Def} {\em (Right/Left-convexity)} A curve in a disk with holes enclosing two distinct sets of adjacent holes as illustrated in Figure~\ref{fig: convex} is called right-convex, and the mirror image of a right-convex curve is called left-convex.  By definition any convex curve enclosing a set of adjacent  holes is both right-convex and left-convex. \end{Def}

 \begin{figure}[ht]  \relabelbox \small {\epsfxsize=1.1in
\centerline{\epsfbox{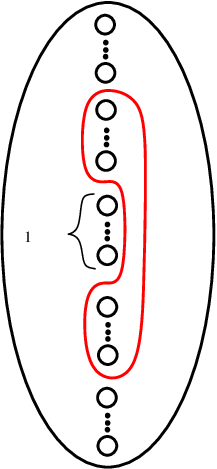}}}\relabel{1}{$0 \leq $}
 \endrelabelbox
\caption{A right-convex curve in a disk with holes.}\label{fig: convex} \end{figure}

{\bf Notation for the rest of the paper:} During the proof, it will be convenient to keep track of the number of wires in our wiring diagrams when talking about marked
points. Thus we will write $x_i^k$ (resp. $y_j^k$) when talking about the marked point $x_i$ (resp. $y_j$) in a wiring diagram with $k$ wires. Although this decoration will make the notation cumbersome, it is necessary for the accuracy of the arguments, but the reader can safely ignore this superscript for the most part in the text below.
Similarly, when talking about curves and half-twists in a disk with holes, it will be convenient to keep track of the
number of holes. For example, we will write $\g_s^k$  when talking about the convex curve $\g_s$
 in the disk $D_k$ with $k$ holes. Moreover, the counterclockwise half-twist $\D(x_i)$ in $D_k$ will be abbreviated by $\D^k_i$ and its inverse, the clockwise half-twist, by $(\D^k_i)^{-1}$. (Fortunately, we will not need to use $\D(y_j)$ in our discussion below by Lemma~\ref{lem: typex}, and hence the notation $\D^k_i$ will not lead to any confusion.) Furthermore,  we will denote the $i$th
hole (with respect to the geometric order from top to bottom) in $D_k$ by $H_i^k$. We will also need  the following definitions.

\begin{Def} \label{def: red} For $2\leq s\leq k$, we denote by $\G_s^k$ the collection of red arcs in $D_k$ shown in Figure \ref{fig: figa}, where we set $\G_1^k := \emptyset$. \end{Def}

 \begin{figure}[ht]  \relabelbox \small {\epsfxsize=1.8in
\centerline{\epsfbox{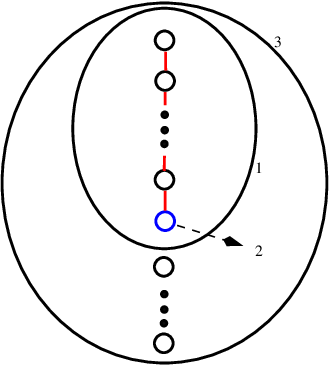}}}\relabel{1}{$\g^k_{s}$}  \relabel{2}{$H^k_s$}  \relabel{3}{$D_k$}
 \endrelabelbox
\caption{ $\G_s^k$ is the collection of red arcs. }\label{fig: figa} \end{figure}

Fix any wiring diagram $\mathcal{W}_\textbf{n}$ with $k$ wires and $k-1$
marked points $x^k_{k-1}, \ldots, x^k_{1}$ (reading from left to right), as in Definition~\ref{def: wd}.

\begin{Def} \label{def: arcs}  For  $2 \leq s\leq k$, let $\rho_s^k$ be the smallest $t \in \{1, \ldots, k-1\}$ such that the convex curve  $\d (x_{t}^k) \subset D_k$ assigned  to $x_{t}^k$ contains
$H_s^k$. For $1\leq s\leq k$ and  $1\leq r\leq k-1$, we define
\begin{align*}
\g_{s,r}^k&:=(\D_r^k)^{-1}\circ(\D_{r-1}^k)^{-1}\circ\cdots\circ(\D_1^k)^{-1}(\g_s^k),  \\
\G_{s,r}^k&:=\begin{cases}(\D_r^k)^{-1}\circ(\D_{r-1}^k)^{-1}\circ\cdots\circ(\D_{\rho_s^k}^k)^{-1}(\G_s^k) & \text{if } \; s \geq 2\;  \text{and } \;r\geq\rho_s^k \\
\G_s^k & \text{otherwise,} \end{cases}
\end{align*}
and set $\g_{s,0}^k:=\g_s^k$ and $\G_{s,0}^k:=\G_s^k$.
 \end{Def}

 Note that $\G_{1,r}^k =\emptyset$ for any $0 \leq r \leq k-1$, by Definition~\ref{def: arcs}.

 \begin{Def} \label{def: psi} We abbreviate $$\Psi_r^k:=(\D_r^k)^{-1}\circ(\D_{r-1}^k)^{-1}\circ\cdots\circ(\D_1^k)^{-1}$$ for $r\geq 1$ and set
$\Psi_0^k=\id$. \end{Def}

\begin{Prop} \label{prop: lobe} Fix any wiring diagram $\mathcal{W}_\textbf{n}$ with $k$ wires and $k-1$
marked points $x^k_{k-1}$, $\ldots,$ $x^k_{1}$ (reading from left to right),   as in Definition~\ref{def: wd}. Then the following three statements hold for any $1\leq s\leq k$, and $1\leq r\leq k-1$:
\begin{enumerate}[\rm(L1)]
\item The curve $\g_{s,r}^k$ is left or right-convex.
\item For $r=\rho_s^k$, the curve $\g_{s,r}^k$ is right-convex, and for $r>\rho_s^k$, if $x_r^k$ contains $\Psi_{r-1}^k(H_s^k)$, then $\g_{s,r}^k$ is left (resp. right) convex if $\g_{s,r-1}^k$ is right (resp. left) convex, otherwise $\g_{s,r}^k$ retains any one-sided convexity of $\g_{s,r-1}^k$, possibly also gaining the other sided convexity.
\item If $s \geq 2$ and $r\geq\rho_s^k$, then $\G_{s,r}^k$ has one of the two forms shown in Figure \ref{fig: figb}.
In particular, the holes enclosed by $\g_{s,r}^k$ can be split into two collections of adjacent holes, which we will
call ``lobes" with the lobe containing $\Psi_r^k(H_s^k)$ being called the ``primary lobe" and the other lobe the ``secondary lobe".
We require $\Psi_r^k(H_s^k)$ to be the innermost hole of the primary lobe and the
intersection of $\G_{s,r}^k$ with some half-plane containing the primary lobe to have precisely one of the two forms illustrated in Figure \ref{fig: figb}. 
\end{enumerate} \end{Prop}

 \begin{figure}[ht]  \relabelbox \small {\epsfxsize=5in
\centerline{\epsfbox{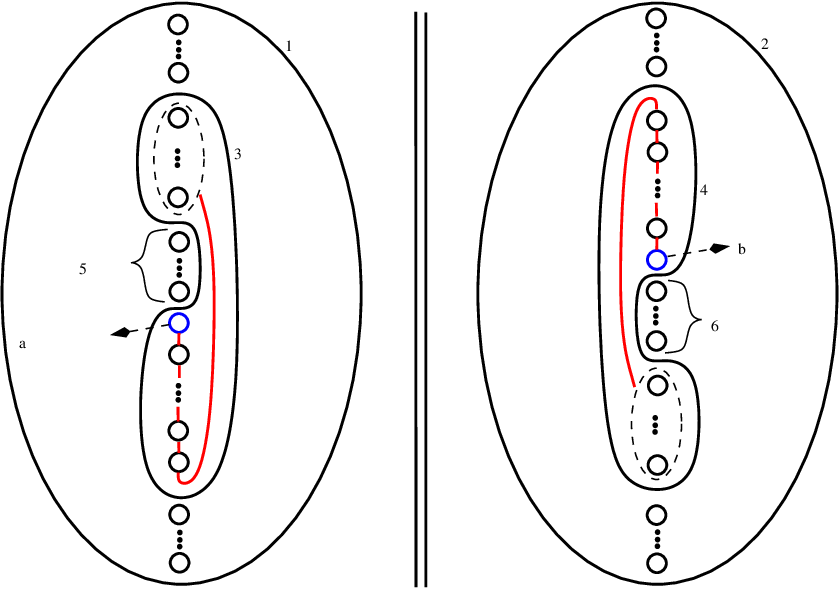}}} \relabel{1}{$D_k$}  \relabel{2}{$D_k$} \relabel{3}{$\g_{s,r}^k$} \relabel{4}{$\g_{s,r}^k$} \relabel{5}{$0 \leq $} \relabel{6}{$\geq 0$} \relabel{a}{$\Psi_r^k(H_s^k)$} \relabel{b}{$\Psi_r^k(H_s^k)$}
 \endrelabelbox
\caption{$\G_{s,r}^k$ is the collection of red arcs in both forms in (L3).}\label{fig: figb} \end{figure}

\begin{Lem} \label{lem: lobetomirror} Proposition~\ref{prop: lobe} implies Proposition~\ref{prop: mirror} {\em (a)}. \end{Lem}

\begin{proof} Lemma~\ref{lem: consec} implies that, for each $1\leq s\leq k$, the top $s$ wires according to their geometric order on the right-hand side of a wiring diagram as described  in Definition~\ref{def: wd} will be {\em consecutive} (perhaps with a different geometric order) on the left-hand side as well. Therefore,  by definition, the convex curve $\d(y^k_s)$ encloses the set of adjacent holes in $D_k$ each of whose order is the same as the local geometric order of one of these $s$ wires on the left-hand side of the diagram.

On the other hand, by definition, the convex curve $\g_s^k$ encloses the top $s$ holes in $D_k$ and the set of images of these holes under $(\D_{k-1}^k)^{-1}  \circ \cdots\circ(\D_1^k)^{-1}$ will be the same as the set of adjacent holes enclosed by $\d(y^k_s)$. To see this, imagine that each wire has a colour and that each hole in the initial copy of $D_k$ has a colour so that the $i$th hole from the top has the same colour as the $i$th wire from the top on the right hand side.  As the wires move from right to left, they will be locally reordered each time a marked points appears in the diagram. Similarly, the clockwise half-twist corresponding to that marked point will reorder the holes on the disk $D_k$. We set up our algorithm so that at each step the colour of each wire remains the same as the  colour of the corresponding hole.

Moreover, for each $1\leq s\leq k$, we know by Proposition~\ref{prop: lobe} that the curve
$$\g_{s,k-1}^k:=(\D_{k-1}^k)^{-1}  \circ \cdots\circ(\D_1^k)^{-1}(\g_s^k)$$ is right or left-convex, but since it encloses a set of {\em adjacent} holes, it must be {\em convex}. Therefore we conclude that for each $1\leq s\leq k$, the convex curve $\g_{s,k-1}^k$ is isotopic to the convex curve  $\d(y^k_s)$.  \end{proof}

\begin{proof}[Proof of Proposition~\ref{prop: lobe}.]  We will prove Proposition~\ref{prop: lobe} by induction on the number of wires in the wiring diagram.
These three statements are vacuously true for a wiring diagram with one wire and no marked points. Now suppose that $k\geq 2$
and these statements hold for any wiring diagram constructed using the blowup algorithm above with $k-1$ wires and
$k-2$ marked points. We will prove that they hold for any wiring diagram with $k$ wires and $k-1$
marked points $x^k_{k-1}, \ldots, x^k_{1}$ (reading from left to right), constructed as in Definition~\ref{def: wd}. Our induction argument naturally splits into several cases. \smallskip

{\bf \underline{Case I (Exterior blowup)}:} This is the easiest case. Suppose that the last wire $w_k$ is inserted into the diagram as a consequence of an exterior blowup so that $w_k$ lies {\em below} all the wires and has no ``interaction" with the other wires. Recall that $w_k$ carries a free marked point $x^k_{k-1}$ which is placed geometrically to the left of all the previous marked points in the diagram.

Consider a fixed embedding of $D_{k-1} \subset D_k$, where $D_{k-1}$ includes the top $k-1$ holes in $D_k$. In other words $D_k$ is obtained from   $D_{k-1}$ by inserting an extra hole, named  $H^k_k$ by our conventions, at the bottom. Under this embedding, for any $1 \leq t \leq k-2$, the convex curve $\d(x^{k-1}_{t})$ in $D_{k-1}$  can be identified with the convex curve $\d(x^{k}_{t})$ in $D_k$, since $x^{k}_{t} = x^{k-1}_{t}$ in the new diagram. Similarly, $\D^k_{t} = \D^{k-1}_{t}$ for any $1 \leq t \leq k-2$, under this embedding and hence it follows that for $1\leq s\leq k-1$ and $1\leq r\leq k-2$ we have $\g_{s,r}^k = \g_{s,r}^{k-1}$ and $\G_{s,r}^k = \G_{s,r}^{k-1}$, which proves by induction,  that statements (L1), (L2) and (L3) hold for the new wiring diagram with $k$ wires, for these cases.

Note that the convex curve $\d(x^k_{k-1})$ is the curve that encloses the last hole $H^k_k \subset D_k$, by definition. Therefore, for  $1\leq s\leq k-1$, the clockwise half-twist $(\D_{k-1}^k)^{-1}$  has no effect on the \emph{convex} curve $\g_{s,k-2}^k$ nor on the collection of arcs $\G_{s,k-2}^k$. Hence statements (L1), (L2) and (L3) hold for $1\leq s\leq k-1$ and $r=k-1$ as well in the new wiring diagram with $k$ wires.

Finally, we observe that $\g_{k,r}^k = \g_{k}^k$ is convex for each $r$, hence (L1) and (L2) automatically hold for $s=k$. Also, $\rho_k^k=k-1$ and
$\G_{k,k-1}^k = (\D_{k-1}^k)^{-1}(\G_k^k)$ has the form shown in Figure \ref{fig: figc}, thus (L3) also holds for $s=k$.  \smallskip

 \begin{figure}[ht]  \relabelbox \small {\epsfxsize=1.5in
\centerline{\epsfbox{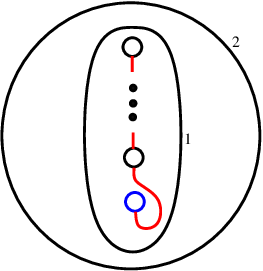}}} \relabel{1}{$\g_{k,k-1}^k$}  \relabel{2}{$D_k$}
 \endrelabelbox
\caption{$\G_{k,k-1}^k$ is the collection of red arcs.}\label{fig: figc} \end{figure}

{\bf \underline{Case II (Interior blowup)}:} Suppose that the last wire $w_k$ is introduced into the diagram as a consequence of an interior blowup at the $i$th term so that $w_k$ is initially right below the $(i+1)$st wire with respect to the {\em geometric ordering} of the wires on the right-hand side of the diagram. Suppose that this $(i+1)$st wire is $w_j$.
Now imagine that we take a step back in our blowup algorithm. In other words, we delete the last wire $w_k$ (and the last marked point $x^k_{k-1}$ and the associated last twisting) from the diagram. At the same time we remove the {\em corresponding hole} from $D_k$ as follows. First we remove the  $(i+2)$nd hole from $D_k$ to obtain the rightmost copy of $D_{k-1}$. As we move from right to left in the wiring  diagram, every time we pass through a marked point, we have a new copy of $D_{k-1}$ obtained by removing from $D_k$ the hole whose order is the same as the local geometric order of the wire $w_k$. During this process,  we will produce several copies of $D_{k-1}$. By reading from right to left, however, all these copies of $D_{k-1}$ can of course be identified with the rightmost copy of $D_{k-1}$ and we use this observation in our induction argument below,
 where we proceed according to three possible cases. \smallskip

{\bf \underline{Case II.A (Interior blowup, $1\leq s \leq i$)}:} Suppose that $1\leq s \leq i$. By hypothesis, $\g_{s,r}^{k-1}$ and  $\G_{s,r}^{k-1}$ satisfy statements (L1), (L2) and (L3) in $D_{k-1}$ for $1 \leq r \leq k-2$. Now since $\g^{k-1}_s$ does not contain the hole $H^{k-1}_{i+1}$, by the assumption that $1\leq s \leq i$, the image $\Psi_{r-1}^{k-1}(H_{i+1}^{k-1})$ is not contained in $\g_{s,r}^{k-1}$ for $1 \leq r \leq k-2$. Therefore, we can insert back the hole we deleted by splitting the image $\Psi_{r-1}^{k-1}(H_{i+1}^{k-1})$ into two adjacent holes. The new hole will be inserted right below $H^{k-1}_{i+1}$ in the rightmost copy of $D_{k-1}$ and it will be inserted right below or above $\Psi_{r-1}^{k-1}(H_{i+1}^{k-1})$ in an alternating fashion every time we pass a marked point  that belongs to the intersection $w_j \cap w_k$. As a result, the superscript $k-1$ can be promoted to $k$, meaning that the curve $\g_{s,r}^{k-1}$ can be viewed as $\g_{s,r}^{k}$ and  $\G_{s,r}^{k-1}$ can be viewed as $\G_{s,r}^{k}$, since we have not modified them by the insertion of the new hole. Hence $\g_{s,r}^{k}$ and $\G_{s,r}^{k}$ satisfy the statements (L1), (L2) and (L3) on $D_{k}$, for $1 \leq r \leq k-2$, as well.

To finish the proof of this case, we only need to argue that $\g_{s,k-1}^{k}$ and $\G_{s,k-1}^{k}$ satisfy the statements (L1), (L2) and (L3). But by the discussion
above,   $\g_{s,k-2}^{k}$ is right or left-convex and it belongs to the subdisk in $D_k$  along which we apply $(\D^k_{k-1})^{-1}$  corresponding to the new marked point $x_{k-1}$, by  Lemma~\ref{lem: consec}.  It follows that $ \g_{s,k-1}^{k} = (\D^{k}_{k-1})^{-1}(\g_{s,k-2}^{k})$ and  $\G_{s,k-1}^{k} = (\D^{k}_{k-1})^{-1}(\G_{s,k-2}^{k})$ satisfy  (L1), (L2) and (L3) as well.  \smallskip

{\bf \underline{Case II.B (Interior blowup, $i+2 \leq s \leq k$)}:} Suppose that  $i+2 \leq s \leq k$.  We check that $\g_{s,r}^k$ and $\G_{s,r}^k$ satisfy
statements (L1), (L2) and (L3) for $1\leq r\leq k-1$. We will proceed by induction on $r$. In the case $r<\rho_s^k$,
the statements are trivial since $\g_{s,r}^k=\g_s^k$ and $\G_{s,r}^k=\G_s^k$ in this case, by Definition~\ref{def: arcs}. If $r=\rho_s^k$, then, after applying the clockwise half-twist $(\D_r^k)^{-1}$ to $\g_{s,r-1}^k=\g_s^k$ and $\G_{s,r-1}^k=\G_s^k$, we see that $\g_{s,r}^k$ and
$\G_{s,r}^k$ have the form given in Figure \ref{fig: figd}. Thus statements (L1), (L2) and (L3) hold in this case also.

 \begin{figure}[ht]  \relabelbox \small {\epsfxsize=2in
\centerline{\epsfbox{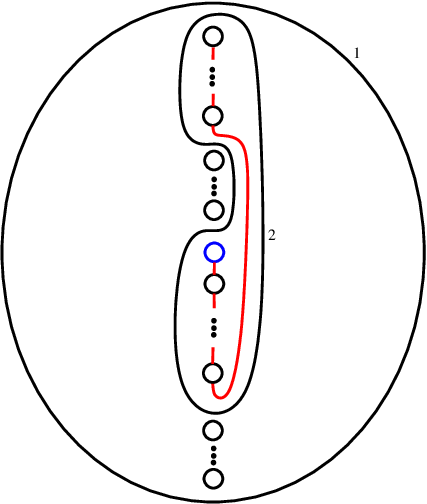}}}  \relabel{2}{$\g_{s,\rho_s^k}^k$}  \relabel{1}{$D_k$}
 \endrelabelbox
\caption{$\G_{s,\rho_s^k}^k$ is the collection of red arcs. }\label{fig: figd} \end{figure}

Now suppose that statements (L1), (L2) and (L3) hold for $k-3\geq r=p\geq\rho_s^k$. We check that the statements
continue to hold for $r=p+1$. For this first note that the convex curve $\d(x_{p+1}^k)$ encloses the hole
$\Psi_p^k(H_s^k)$ if and only if the convex curve $\d(x_{p+1}^{k-1})$ encloses the hole $\Psi_p^{k-1}(H_{s-1}^{k-1})$.
Indeed, if $s>i+2$, then the wire in geometric position $s$ on the right hand side of the wiring diagram
is wire $w_l$ for some $l<k$, since wire $w_k$ is in geometric position $i+2$ on the right hand side. If we take a step back in
our blowup algorithm, then wire $w_l$ will have geometric position $s-1$ on the right hand side of the wiring
diagram, now with $k-1$ wires. For $s>i+2$, the result claimed now follows  from the general fact that $\d(x_{p+1}^m)$ encloses the hole
$\Psi_p^m(H_t^m)$ if and only if the wire in geometric position $t$ on the right hand side of the diagram
passes through the marked point $x_{p+1}^m$.
If $s=i+2$, then the wire in geometric position $s$ on the right hand side of the wiring diagram
will be wire $w_k$. Since wire $w_k$ passes through each marked point $x_q^k$ that wire $w_j$ passes through for $1\leq q\leq k-2$, otherwise remaining
parallel to $w_j$, arguing as above, we obtain the same result for $s=i+2$.

By hypothesis of the induction on $r$, it now
follows that the curve $\g_{s,p}^k$ will be right or left-convex according to whether $\g_{s-1,p}^{k-1}$ is
right or left-convex. Furthermore, the holes that $\g_{s,p}^k$ encloses can be obtained from the holes that
$\g_{s-1,p}^{k-1}$ encloses by splitting the hole $\Psi_p^{k-1}(H_{i+1}^{k-1})$ into two adjacent holes.

In a similar way, the holes that $\d(x_{p+1}^k)$ encloses can be obtained from the holes that $\d(x_{p+1}^{k-1})$ encloses by splitting the hole $\Psi_p^{k-1}(H_{i+1}^{k-1})$ into two adjacent holes. As a consequence, the curve
$\g_{s,p+1}^k=(\D^k_{p+1})^{-1}(\g_{s,p}^k)$ satisfies (L1) and (L2) and the holes that $\g_{s,p+1}^k$ encloses can be obtained from the holes that $\g_{s-1,p+1}^{k-1}$ encloses by splitting the hole $\Psi_{p+1}^{k-1}(H_{i+1}^{k-1})$ into two adjacent holes.

We now show that $\G_{s,p+1}^k$ satisfies (L3) by considering the cases that $\d(x_{p+1}^k)$ encloses and does not enclose the hole
$\Psi_p^k(H_s^k)$  separately. First suppose that $\d(x_{p+1}^k)$ does not enclose
the hole $\Psi_p^k(H_s^k)$. Then $\d(x_{p+1}^{k-1})$
does not enclose the hole $\Psi_p^{k-1}(H_{s-1}^{k-1})$. Assume that $\g_{s-1,p}^{k-1}$
is right-convex; the case that $\g_{s-1,p}^{k-1}$ is left-convex is similar. Then $\g_{s-1,p+1}^{k-1}$ is also
right-convex and we have the following possibilities for $\d(x_{p+1}^{k-1})$:

(i) $\d(x_{p+1}^{k-1})$ is disjoint from $\g_{s-1,p}^{k-1}$. In this case $\d(x_{p+1}^{k-1})$ cannot enclose a subset of holes from the primary lobe of
$\g_{s-1,p}^{k-1}$, since otherwise $\G_{s-1,p+1}^{k-1}$ would fail to satisfy (L3). It follows in this case that $\d(x_{p+1}^k)$ is disjoint
from $\g_{s,p}^k$ and does not enclose any subset of holes from the primary lobe of $\g_{s,p}^k$.
Thus $\G_{s,p+1}^k$ satisfies (L3).

(ii) $\d(x_{p+1}^{k-1})$ intersects $\g_{s-1,p}^{k-1}$ and encloses at least one hole below the primary lobe of $\g_{s-1,p}^{k-1}$.
In this case $\G_{s-1,p+1}^{k-1}$ would fail to satisfy (L3), which contradicts our induction hypothesis. Hence this case cannot occur.

(iii) $\d(x_{p+1}^{k-1})$ intersects $\g_{s-1,p}^{k-1}$ and encloses at least one hole above the secondary lobe of $\g_{s-1,p}^{k-1}$.
In this case, if $\g_{s-1,p}^{k-1}$ does not contain all the holes between the two lobes, $\g_{s-1,p+1}^{k-1}$ would fail to be
one-sided convex, which contradicts (L1), and if it does contain all the holes between the two lobes, then $\g_{s-1,p+1}^{k-1}$ would become left-convex, but not convex, which contradicts (L2). Therefore, by  the induction hypothesis these cases cannot occur.

(iv) $\d(x_{p+1}^{k-1})$ intersects $\g_{s-1,p}^{k-1}$ and encloses at least one hole below the secondary lobe of $\g_{s-1,p}^{k-1}$
but no hole above it.
If $\d(x_{p+1}^{k-1})$ does not enclose all the holes contained in the secondary lobe of $\g_{s-1,p}^{k-1}$, then either $\g_{s-1,p+1}^{k-1}$
would have more than two ``lobes" or $\Psi_r^{k-1}(H_{s-1}^{k-1})$ would not be the innermost hole of the primary lobe. Both of these cases contradict the induction hypothesis hence they cannot occur. Thus the only
possibility in this case is that $\d(x_{p+1}^{k-1})$ encloses all the holes contained in the secondary lobe and at least one
hole below it.
This case is illustrated in Figure~\ref{fig: fige}(a). In this case $\d(x_{p+1}^k)$ will enclose all holes in the secondary lobe of $\g_{s,p}^k$
and enclose at least one hole below it. Hence $\G_{s,p+1}^k$ will satisfy (L3) in this case also.
This concludes the analysis for the case $\d(x_{p+1}^k)$ not enclosing the hole $\Psi_p^k(H_s^k)$.

Now suppose that $\d(x_{p+1}^k)$ encloses the hole $\Psi_p^k(H_s^k)$. Then $\d(x_{p+1}^{k-1})$ encloses the hole $\Psi_p^{k-1}(H_{s-1}^{k-1})$.
Suppose that $\g_{s,p}^k$ is right-convex; again the left-convex case is similar. Then $\g_{s-1,p}^{k-1}$ is also right-convex.
By the induction hypothesis, the image of $\g_{s-1,p}^{k-1}$ under the clockwise half-twist about $\d(x_{p+1}^{k-1})$
must be left-convex and the image of $\G_{s-1,p}^{k-1}$ must continue to satisfy (L3). In this case, by an argument similar to what we used above,   it can be checked that  $\d(x_{p+1}^{k-1})$  must enclose all the holes of the secondary lobe (and no
holes above it) and any number of holes below $\Psi_p^{k-1}(H_{s-1}^{k-1})$; this situation is illustrated in
Figure~\ref{fig: fige}(b). 
Since, by splitting vertically the hole $\Psi_p^{k-1}(H_{i+1}^{k-1})$ into two adjacent holes, we obtain $\g_{s,p}^k$ and $\d(x_{p+1}^k)$
from $\g_{s-1,p}^{k-1}$ and $\d(x_{p+1}^{k-1})$, respectively, it follows that $\g_{s,p}^k$ and $\d(x_{p+1}^k)$ will have the same form as
 $\g_{s-1,p}^{k-1}$ and $\d(x_{p+1}^{k-1})$, that is, $\d(x_{p+1}^k)$ will enclose $\g_{s,p}^k$ or it will enclose all the holes of the secondary
 lobe (and no holes above it) and any number of holes below $\Psi_p^k(H_s^k)$.
 It is now clear that, after applying a
 clockwise half-twist about $\d(x_{p+1}^k)$  to  $\G_{s,p}^k$, $\G_{s,p+1}^k$ will continue to satisfy (L3). Although we did not draw it in Figure~\ref{fig: fige}(b), it is possible that $\d(x_{p+1}^{k-1})$ contains holes below the primary lobe, and hence contains  $\g_{s-1,p}^{k-1}$,  in which case (L3) is satisfied trivially by $\G_{s,p+1}^k$.

We have thus checked that $\g_{s,r}^k$ and $\G_{s,r}^k$ satisfy (L1), (L2) and (L3) for $r\leq k-2$. We now check
that they satisfy  (L1), (L2) and (L3) for $r=k-1$ also.

\begin{figure}[ht]  \relabelbox \small {\epsfxsize=3.5in
\centerline{\epsfbox{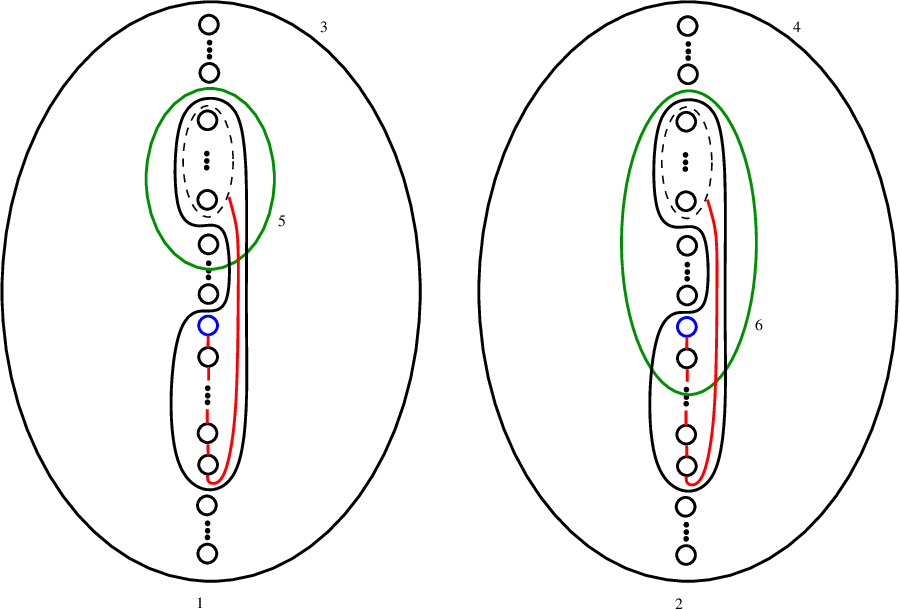}}}
\relabel{1}{(a)}  \relabel{2}{(b)} \relabel{3}{$D_{k-1}$} \relabel{4}{$D_{k-1}$} \relabel{5}{$\d(x_{p+1}^{k-1})$} \relabel{6}{$\d(x_{p+1}^{k-1})$}
 \endrelabelbox
\caption{Some possibilities for $\d(x_{p+1}^{k-1})$.  }\label{fig: fige} \end{figure}

As $\g_{s-1,k-2}^{k-1}$ will be in the left-most copy of $D_{k-1}$, by
Lemma \ref{lem: consec}, all the holes enclosed
by $\g_{s-1,k-2}^{k-1}$ will be adjacent and the hole $\Psi_{k-2}^{k-1}(H_{s-1}^{k-1})$ will be either at the top or the
bottom of these holes. Being one-sided convex (by the induction hypothesis), the curve $\g_{s-1,k-2}^{k-1}$ must be
convex and hence the pair  $(\g_{s-1,k-2}^{k-1},\G_{s-1,k-2}^{k-1})$ must be as in Figure~\ref{fig: figf}.  Thus the curve $\g_{s,k-2}^k$ must also be convex as the holes it encloses are obtained from the holes that $\g_{s-1,k-2}^{k-1}$ encloses by splitting $\Psi_{k-2}^{k-1}(H_{i+1}^{k-1})$ into two adjacent holes. As the holes enclosed by $\g_{s,k-1}^k$ must be adjacent with the hole $\Psi_{k-1}^k(H_s^k)$ at one end
(again, by Lemma \ref{lem: consec}), it follows that $\d(x_{k-1}^k)$ must be disjoint from
$\g_{s,k-2}^k$. Hence the curve $\g_{s,k-1}^k$ will be convex and statements (L1) and (L2) will hold for $r=k-1$ also.

To see that (L3) also holds for $r=k-1$, first suppose that $s>i+2$. Then $\Psi_{k-2}^k(H_s^k)$ will already be at the top or bottom of the holes enclosed by $\g_{s,k-2}^k$. As the last marked point $x_{k-1}^k$ only involves the wires having the geometric positions $\{t\,|\,1\leq t\leq i+2, t\neq i+1\}$ on the right hand side, the last clockwise half-twist $(\D_{k-1}^k)^{-1}$ will be about a subdisk that does not enclose $\Psi_{k-2}^k(H_s^k)$. It easily follows that $\G_{s,k-1}^k$ will satisfy (L3).

Now suppose that $s=i+2$. Then the hole $\Psi_{k-2}^k(H_s^k)$ will be either one below the top hole or one above the
bottom hole of $\g_{s,k-2}^k$. The last clockwise half-twist will be about a subdisk that encloses all the holes of $\g_{s,k-2}^k$ except the hole
$\Psi_{k-2}^k(H_{i+1}^k)$, which will be at the top or the bottom. Again it follows that $\G_{s,k-1}^k$ will satisfy (L3).
This completes the proof of {Case~II.B}.

\begin{figure}[ht]  \relabelbox \small {\epsfxsize=3.5in
\centerline{\epsfbox{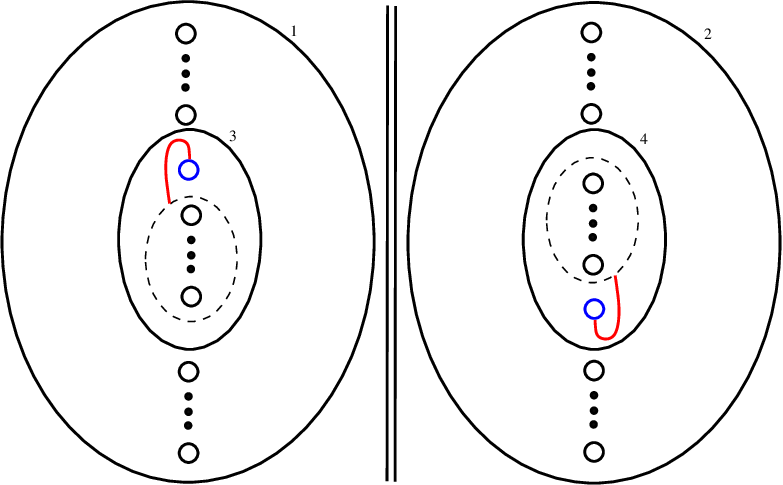}}} \relabel{3}{$\g_{s-1,k-2}^{k-1}$}  \relabel{4}{$\g_{s-1,k-2}^{k-1}$} \relabel{1}{$D_{k-1}$} \relabel{2}{$D_{k-1}$}
\endrelabelbox
\caption{One arc of $\G_{s-1,k-2}^{k-1}$ is shown in red.  }\label{fig: figf} \end{figure}

\smallskip

{\bf \underline{Case II.C (Interior blowup, $s=i+1$)}:} Suppose that $s=i+1$. By the proof of {Case~II.B}, we know that $\g_{i+2,r}^k$ and $\G_{i+2,r}^k$ satisfy
(L1), (L2) and (L3) for $1\leq r\leq k-1$. Note that for $1\leq r\leq k-2$, the holes $\Psi_r^k(H_{i+1}^k)$ and $\Psi_r^k(H_{i+2}^k)$
will be adjacent, since the wires $w_j$ and $w_k$, having geometric positions $i+1$ and $i+2$, respectively, on the
right hand side of the diagram, will remain consecutive up to the marked point $x_{k-1}$. Thus hole $\Psi_r^k(H_{i+1}^k)$
will be in the primary lobe of $\g_{i+2,r}^k$ for $1\leq r\leq k-2$. It is now clear that $\g_{i+1,r}^k$ is given by isotoping $\g_{i+2,r}^k$
over the hole $\Psi_r^k(H_{i+2}^k)$ from the side dictated by $\G_{i+2,r}^k$ for $1\leq r\leq k-2$; see Figure~\ref{fig: figg}.
It follows that statements (L1), (L2) and (L3) hold for $\g_{i+1,r}^k$ and $\G_{i+1,r}^k$ for $1\leq r\leq k-2$. We point out that it is at this stage in the induction that we make essential use of the fact that $\G_{i+2,r}^k$ satisfies (L3) to obtain that $\g_{i+1,r}^k$ satisfies (L1).

\begin{figure}[ht]  \relabelbox \small {\epsfxsize=3.5in
\centerline{\epsfbox{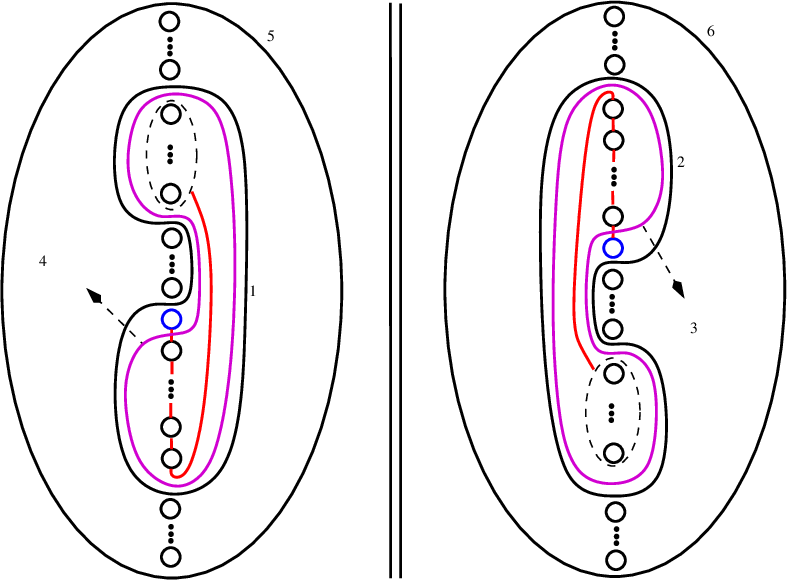}}}   \relabel{1}{$\g_{i+2,r}^k$}   \relabel{2}{$\g_{i+2,r}^k$}  \relabel{3}{$\g_{i+1,r}^k$} \relabel{4}{$\g_{i+1,r}^k$} \relabel{5}{$D_k$} \relabel{6}{$D_k$} 
\endrelabelbox
\caption{The case when $1 \leq r\leq k-2$.}\label{fig: figg} \end{figure}

For $r=k-1$, we argue as follows: Since $\g_{i+1,k-2}^{k-1}$ will be convex with the hole $\Psi_{k-2}^{k-1}(H_{i+1}^{k-1})$
at one end, the curve $\g_{i+2,k-2}^{k}$ will also be convex with the hole $\Psi_{k-2}^k(H_{i+1}^k)$ at one end and the hole $\Psi_{k-2}^k(H_{i+2}^k)$ adjacent to it, both in the primary lobe of $\g_{i+2,k-2}^{k}$. By the discussion in the previous paragraph, it follows that the curve $\g_{i+1,k-2}^k$ must have one of the two forms given in  Figure~\ref{fig: figh}. It easily follows that
$\g_{i+1,k-1}^k$ and $\G_{i+1,k-1}^k$ satisfy statements (L1), (L2) and (L3). This completes the proof of {Case~II.C}. \end{proof}

\begin{figure}[ht]  \relabelbox \small {\epsfxsize=4.5in
\centerline{\epsfbox{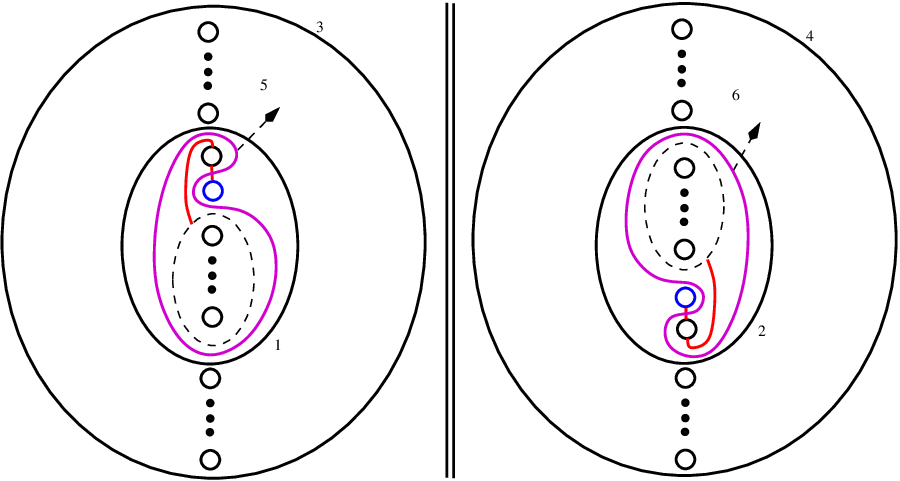}}} \relabel{1}{$\g_{i+2,k-2}^k$}  \relabel{2}{$\g_{i+2,k-2}^k$} \relabel{3}{$D_k$} \relabel{4}{$D_k$} \relabel{5}{$\g_{i+1,k-2}^k$} \relabel{6}{$\g_{i+1,k-2}^k$}
\endrelabelbox
\caption{The case when $r = k-1$.}\label{fig: figh} \end{figure}

\subsubsection{The case of $\a$-curves:} \label{subsec: alpha} We reformulate our claim in Proposition~\ref{prop: mirror} (b) about $\a$-curves as Proposition~\ref{prop: alpha} below, where have we replaced $\mathcal{W}_\textbf{n} (\textbf{m})$ by $\mathcal{W}_\textbf{n}$, since the extension from $\mathcal{W}_\textbf{n}$ to $\mathcal{W}_\textbf{n} (\textbf{m})$ is irrelevant.  Recall that $\overline{\a}$ denotes the mirror image of a given curve $\a$. In the following, we will decorate each curve with a superscript to indicate the number of holes in the disk in which they are embedded. For example,  we will use  $\a^k_t$ to indicate that we are talking about the curve $\a_t$ (see Section~\ref{sec: stab}) in $D_k$. Similarly, we  will decorate each marked point in a wiring diagram with a superscript to indicate the number of wires in the diagram.

 \begin{Prop} \label{prop: alpha} Let $\mathcal{W}_\textbf{n}$ be a wiring diagram  with $k \geq 2$ wires and $k-1$
marked points $x^k_{k-1},\ldots,x^k_{1}$ (reading from left to right) constructed as in Definition~\ref{def: wd} with respect to some blowup sequence. Then,  for each $1 \leq t \leq k-1$, the vanishing cycle $V(x^k_{t})$ associated to the marked point $x^k_{t}$ via the method of Plamenevskaya and  Starkston, is the mirror image  $\overline{\a}^k_t$ of the curve ${\a}^k_t$ obtained by the stabilization algorithm described in Section~\ref{sec: stab} with respect to the same blowup sequence.   \end{Prop}

\begin{proof}  According to Definition~\ref{def: cycle}, $V(x^k_{1}) =\d(x^k_1)$ and $V(x^k_{t}) = \D^k_1 \circ \cdots \circ \D^k_{t-1} (\d(x^k_t))$ for  $2 \leq t \leq k-1$. Therefore, we need to verify that $\d(x^k_1) =  {\a}^k_1$ (both curves are convex and $ \overline{\a}^k_1= {\a}^k_1$) and $\D^k_1 \circ \cdots \circ \D^k_{t-1} (\d(x^k_t)) =  \overline{\a}^k_t, $   for  $2 \leq t \leq k-1$.

For $k=2$, the wiring diagram has only two parallel wires and one marked point $x^2_1$ on the bottom wire, corresponding to the initial blowup $(0) \to (1,1)$. The statement holds for this case since $\d(x^2_1) =  {\a}^2_1$.  Now suppose that $k\geq 3$
and the statement holds for any wiring diagram with $k-1$ wires and $k-2$ marked points, constructed as in Definition~\ref{def: wd} with respect to some blowup sequence. We will prove  that the statement holds for any wiring diagram with $k$ wires and
$k-1$ marked points obtained by inserting one more wire and a marked point corresponding to the new blowup.  \smallskip

{\bf \underline{Case I (Exterior blowup)}:} This is the easiest case.  Suppose that the new wire $w_k$ is inserted into the diagram with $k-1$ wires as a consequence of an exterior blowup so that $w_k$ lies {\em below} all the wires and has no ``interaction" with the other wires.  Recall that $w_k$ carries a free marked point $x^k_{k-1}$ which is placed geometrically to the left of all the previous marked points in the diagram.

Consider a fixed embedding of $D_{k-1} \subset D_k$, where $D_{k-1}$ includes the top $k-1$ holes in $D_k$. In other words $D_k$ is obtained from   $D_{k-1}$ by inserting an extra hole, named  $H^k_k$ by our conventions, at the bottom. It is clear that for any $1 \leq t \leq k-2$, the convex curve $\d (x^{k-1}_t) \subset D_{k-1}$ can be identified with the convex curve  $\d (x^{k}_t) \subset D_k$ and hence   $\D^k_{t} = \D^{k-1}_{t}$ for any $1 \leq t \leq k-2$, under this embedding. Similarly, for any $1 \leq t \leq k-2$, ${\a}^k_t$ can be identified with ${\a}^{k-1}_t$, by our stabilization algorithm in Section~\ref{sec: stab}.

By induction, the property we want to verify holds for the wiring diagram with $k-1$ wires before we insert $w_k$, i.e., we have $\d(x^{k-1}_1) =  {\a}^{k-1}_1$ and $\D^{k-1}_1 \circ \cdots \circ \D^{k-1}_{t-1} (\d(x^{k-1}_t)) =  \overline{\a}^{k-1}_t, $   for  $2 \leq t \leq k-2$. Under the embedding above, these can be upgraded to $\d(x^k_1) =  {\a}^k_1$  and $\D^k_1 \circ \cdots \circ \D^k_{t-1} (\d(x^k_t)) =  \overline{\a}^k_t, $   for  $2 \leq t \leq k-2$, by simply replacing the superscript $k-1$ with $k$. The key point here is that the composition $\D^k_1 \circ \cdots \circ \D^k_{t-1}$ takes place in the fixed embedded disk $D_{k-1} \subset D_k$.

To finish the proof of this case, we only have to verify the statement for $t=k-1$. Note that $\d (x^k_{k-1}) = \a^k_{k-1}$, which by definition,  is the convex curve that encloses the last hole $H^k_k \subset D_k$. We observe that $$\D^k_1 \circ \cdots \circ \D^k_{k-2} (\d (x^k_{k-1})) = \a^k_{k-1}$$  so that the ``last" vanishing cycle is  $\a^k_{k-1} = \overline{\a}^k_{k-1}$, which is consistent with our stabilization algorithm in Section~\ref{sec: stab}. \smallskip

{\bf \underline{Case II (Interior blowup)}:}  Now, suppose that the new wire $w_k$ is introduced into the diagram with $k-1$ wires as a consequence of an interior blowup at the $i$th term so that $w_k$ is initially right below the $(i+1)$st wire with respect to the {\em geometric ordering} of the wires on the right-hand side of the diagram.  Suppose this $(i+1)$st wire is $w_j$. Our proof below splits into two subcases:  the case where $1 \leq t \leq k-2$, and the case $t = k-1$. \smallskip

{\bf \underline{Case II.A (Interior blowup, $1 \leq t \leq k-2$)}:}  Now, imagine that we take a step back in our blowup algorithm. In other words, we delete the last wire $w_k$ (and the last marked point $x^k_{k-1}$ and the associated last twisting) from the diagram. At the same time we remove the {\em corresponding hole} from $D_k$ as follows. First we remove the  $(i+2)$nd hole from $D_k$ to obtain the rightmost copy of $D_{k-1}$. As we move from right to left in the wiring  diagram, every time we pass through a marked point, we have a new copy of $D_{k-1}$ obtained by removing from $D_k$ the hole whose order is the same as the local geometric order of the wire $w_k$. All these copies of $D_{k-1}$ can of course be identified with the rightmost copy of $D_{k-1}$, in which by induction, we have $$V(x^{k-1}_{t}) := \D^{k-1}_1 \circ \cdots \circ \D^{k-1}_{t-1} (\d(x^{k-1}_t)) = \overline{\a}^{k-1}_t $$ for  $2 \leq t \leq k-2$ and $\d(x^{k-1}_1) =  {\a}^{k-1}_1$. We claim that we can upgrade the superscript $k-1$ to $k$ in the previous sentence by splitting the $(i+1)$st hole in (the rightmost copy of)  $D_{k-1}$, to create a new hole right below it, and hence identifying the new disk as $D_{k}$ and promoting the all the relevant curves from $D_{k-1}$ to $D_{k}$. Note that this is precisely how ${\a}^{k}_t$ is obtained from  ${\a}^{k-1}_t$ by the stabilization algorithm  in Section~\ref{sec: stab}. We just need to show that $V(x^{k-1}_t)$ can be promoted to $V(x^{k}_t)$ in the same manner, for each $1 \leq t \leq k-2$.    This is easy to see for $t=1$ since the convex curve $\d (x^{k-1}_1)$ can be promoted to the convex curve $\d (x^{k}_1)$ by inserting a hole in $D_{k-1}$ either enclosed by $\d (x^{k}_1)$ or not depending on whether $x^{k}_1$ (which is in fact the same marked point $x^{k-1}_1$) belongs to $w_j$ or not.

For $2 \leq t \leq k-2$,  we argue as follows. In the aforementioned copies of $D_{k-1}$, there is a hole whose order is the same as the local geometric order of the wire $w_j$. Since, by the blowup algorithm,  the wires $w_j$ and $w_k$ are geometrically consecutive throughout the diagram (except when they intersect) the corresponding two holes will be adjacent in $D_k$, interchanging their relative order at each intersection point of $w_j$ and $w_k$ (which is indeed a marked point in the diagram). Moreover,  for  each $1 \leq t \leq k-2$, the convex curve $\d(x^k_{t})$ assigned to $x^k_{t}$ will either enclose both or neither of these holes, depending on whether $x^k_t$ belongs to $w_j$ or not.  Therefore, for each $2 \leq t \leq k-2$, any one of the curves $$\d(x^k_t),\; \D^k_{t-1} (\d(x^k_t)), \; \ldots, \; \D^k_1 \circ \cdots \circ \D^k_{t-1} (\d(x^k_t))$$ obtained iteratively by applying counterclockwise half-twists starting from  the convex curve $\d(x^k_t)$, will either enclose both or neither of these two holes.  This implies that we could just upgrade the curves $$\d(x^{k-1}_t),\;  \D^{k-1}_{t-1} (\d(x^{k-1}_t)), \; \ldots, \;  \D^{k-1}_1 \circ \cdots \circ \D^{k-1}_{t-1} (\d(x^{k-1}_t))$$ from $D_{k-1}$ to $D_k$ by inserting a new hole (corresponding to the local geometric position of the new wire $w_k$) in each copy of $D_{k-1}$. Note that the new hole corresponding to the new wire $w_k$ can be viewed as being obtained by splitting the hole corresponding to $w_j$ at each step. We observe that this new hole will appear right below the $(i+1)$st hole in the rightmost copy of $D_{k-1}$ giving $D_k$, and these two holes in question are enumerated as $H^k_{i+1}$ and $H^k_{i+2}$ in the rightmost copy of $D_k$. This ``splitting" of the $(i+1)$st hole is the crux of the matter in our stabilization algorithm in  Section~\ref{sec: stab}.

The upshot of this discussion is that for  $1 \leq t \leq k-2$, the vanishing cycle $V( x^k_{t})$  is the curve $\overline{\a}^k_t$, which is in fact nothing but $\overline{\a}^{k-1}_t$ upgraded to $D_k$ from $D_{k-1}$, in a manner consistent with our stabilization algorithm in  Section~\ref{sec: stab}. Finally, the case $t =  k-1$ will be treated separately below. \smallskip

{\bf \underline{Case II.B (Interior blowup, $t=k-1$)}:}  To finish the proof of Proposition~\ref{prop: alpha}, we just have to verify that the ``last" vanishing cycle $V(x^k_{k-1})$  is the curve $\overline{\a}^k_{k-1}$. Equivalently, we need to verify that $\Psi^k_{k-2} ( \overline{\a}^k_{k-1}) = \d(x^k_{k-1})$, where $\Psi^k_{k-2} = (\D_{k-2}^{k})^{-1} \circ \cdots \circ (\D_1^{k})^{-1}$, by Definition~\ref{def: psi}. Here, {\em instead of induction}, we will rather use the statement in Proposition~\ref{prop: mirror}~(a) for $\g_{i+1}$, which we already proved.

The standing assumption in {Case II} is that the new wire $w_k$ is introduced into the diagram with $k-1$ wires as a consequence of an interior blowup at the $i$th term so that $w_k$ is initially right below the $(i+1)$st wire with respect to the {\em geometric ordering} of the wires on the right-hand side of the diagram. Suppose that this $(i+1)$st wire, labelled $w_j$, has an odd number of marked points. The case with an even number of points is very similar and is left to the interested reader.

On the left in Figure~\ref{fig: ex10a}, we depict the convex curves $ \d(y^k_{i+1})$ and $\d(x^k_{k-1})$. This configuration can be deduced from the proof of Lemma~\ref{lem: consec} for the case when $w_j$ has an odd number of marked points. For $s=i+1$, Proposition~\ref{prop: mirror} (a) implies that
$$ \d(y^k_{i+1})=   \Psi^k_{k-1} (\g^k_{i+1}) = (\D^k_{k-1})^{-1} \Psi^k_{k-2} (\g^k_{i+1})$$ and we set $(\g^k_{i+1})' = \D^k_{k-1}(\d(y^k_{i+1}) ) = \Psi^k_{k-2} (\g^k_{i+1})$, which is depicted on the right in Figure~\ref{fig: ex10a}.

\begin{figure}[ht]  \relabelbox \small {\epsfxsize=4.5in
\centerline{\epsfbox{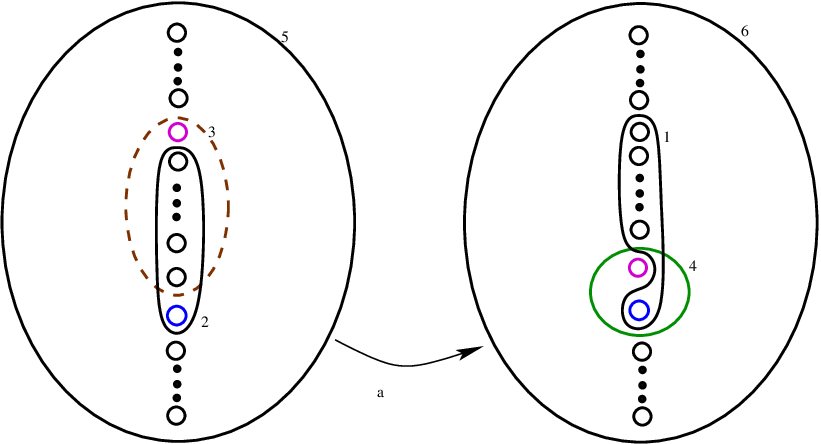}}}
 \relabel{2}{$\d(y^k_{i+1})$} \relabel{a}{$\D^k_{k-1}$} \relabel{1}{$(\g^k_{i+1})'$} \relabel{3}{$\d(x^k_{k-1})$} \relabel{4}{$\b^k_{\psi(i+2)}$} \relabel{5}{$D_k$}  \relabel{6}{$D_k$}
 \endrelabelbox
\caption{The curve $(\g^k_{i+1})' = \D^k_{k-1}(\d(y^k_{i+1}))$. }\label{fig: ex10a} \end{figure}

Let $\b^k_{i+1}$ be the convex curve in  Figure~\ref{fig: ex11} enclosing the holes $H^k_{i+1}$ and  $H^k_{i+2}$ and let $\D^k_{\b_{i+1}}$ denote the counterclockwise half-twist in the subdisk bounded by  $\b^k_{i+1}$. Note that in $D_k$, we have $\overline{\a}^k_{k-1}=\D^k_{\b_{i+1}} (\g^k_{i+1})$ as one can see from Figure~\ref{fig: ex11}.

\begin{figure}[ht]  \relabelbox \small {\epsfxsize=2in
\centerline{\epsfbox{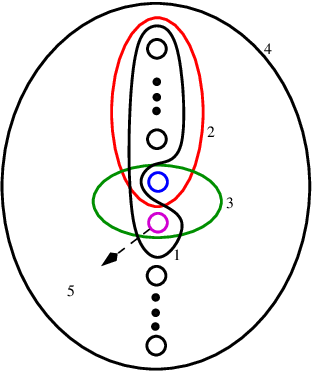}}} \relabel{1}{$\overline{\a}^k_{k-1}$} \relabel{2}{$\g^k_{i+1}$} \relabel{3}{$\b^k_{i+1}$} \relabel{4}{$D_k$}  \relabel{5}{$H^k_{i+2}$}
 \endrelabelbox
\caption{The curves $\g^k_{i+1}$, $\b^k_{i+1}$, and  $\overline{\a}^k_{k-1}$ in $D_k$.}\label{fig: ex11} \end{figure}

Hence we have
\begin{equation}\label{eq1}
\Psi^k_{k-2}(\overline{\a}^k_{k-1} ) =  \Psi^k_{k-2} \circ \D^k_{\b_{i+1}} (\g^k_{i+1})
= \D^k_{\b_{\psi(i+2)}} \circ \Psi^k_{k-2} (\g^k_{i+1}) =  \D^k_{\b_{\psi(i+2)}} ((\g^k_{i+1})'), \end{equation}
 where $\b_{\psi(i+2)}^k$ is defined as follows: The image of the curve $\b^k_{i+1}$ under   $\Psi^k_{k-2}$ is the convex curve  $\b^k_{\psi(i+2)}$ (depicted on the right in Figure~\ref{fig: ex10a}) enclosing two adjacent holes whose geometric orders are given as   $ \{\psi(i+2), \psi(i+2)+1\}$.  Here $\D^k_{\b_{\psi(i+2)}}$ denotes the counterclockwise half-twist in the subdisk bounded by  $\b^k_{\psi(i+2)}$.

Note that the order of the purple coloured hole depicted in the copy of $D_k$ carrying $(\g^k_{i+1})'$  in Figure~\ref{fig: ex10a} is $\psi (i+2)$, and the blue hole right below it has order $\psi (i+2) +1$. In fact,  $\psi(i+2)$ is the local geometric order of the wire $w_k$, and $\psi(i+2)+1$ is the local geometric order of the wire $w_j$, before the last twisting. Recall that these two wires are always geometrically consecutive and they swap their order each time they intersect.

In the second equality in \eqref{eq1} above, we used the fact that
$$\Psi^k_{k-2} \circ \D^k_{\b_{i+1}}  \circ  (\Psi^k_{k-2})^{-1} = \D^k_{\b_{\psi(i+2)}}$$ which can be easily verified.

Finally, applying $\D^k_{\b_{\psi(i+2)}}$ (on the subdisk bounded by $\b^k_{\psi(i+2)}$ enclosing only the blue and the purple holes)  to $(\g^k_{i+1})'$  in Figure~\ref{fig: ex10a} we get the dashed curve $\d(x^k_{k-1})$. Therefore, coupled with \eqref{eq1}, we conclude that  $ \Psi^k_{k-2}(\overline{\a}^k_{k-1} )= \d(x^k_{k-1})$. \end{proof}

\subsection{Proofs of the corollaries and some examples}

\begin{proof}[Proof of Corollary~\ref{cor: arrangement}] According to Theorem~\ref{thm: main}, for each Stein filling $W_{p,q}(\textbf{n})$ of the contact $3$-manifold $(L(p,q), \xi_{can})$, there is an (unbraided) wiring diagram $\mathcal{W}_\textbf{n}({\textbf{m}})$, constructed as in Definition~\ref{def: wd}, based on the  blowup sequence $(0) \to (1,1) \to \cdots \to \textbf{n}$ and $\textbf{m} = \textbf{b}-\textbf{n}$, whose associated planar Lefschetz fibration $f \colon  W \to D^2$ obtained by the method of Plamenevskaya and  Starkston \cite[Section 5.2]{ps}  is equivalent to the one constructed by the authors \cite{bo}.  Moreover, using  \cite[Proposition 5.5]{ps}, the wiring diagram $\mathcal{W}_\textbf{n}({\textbf{m}})$, viewed as a collection of intersecting curves in $\bfr \times \bfc$, can be extended to an explicit  collection of symplectic graphical disks $\G_1, \ldots , \G_n$  in $\bfc \times \bfc$, with marked points $p_1, \ldots, p_m \in \bigcup_i \G_i$ including all the intersection points of these disks. Note that we can assume that each intersection of these disks is positive and transverse, and non-intersection points are allowed as free marked points.   Furthermore,  by  \cite[Proposition 5.6]{ps},  the Stein filling $W_{p,q}(\textbf{n})$ is supported by the restriction of the Lefchetz fibration $$\pi_{x} \circ \Pi \colon \bfc^2 \# m \cpb \setminus (\widetilde{\G}_1 \cup\cdots\cup  \widetilde{\G}_n) \to \bfc$$ to a Milnor ball (to give compact fibres), where $\pi_x \colon \bfc^2 \to \bfc$  denotes the projection onto the first coordinate,  $\widetilde{\G}_1,  \ldots, \widetilde{\G}_n$ are the proper transforms of $\G_1,  \ldots, \G_n$, and $\Pi$ is the blowup map. Finally, the Lefschetz fibrations $f$ and $\pi_{x} \circ \Pi$ are equivalent by the discussion in  \cite[Section 5.4]{ps}. Note that the last statement in Corollary~\ref{cor: arrangement} follows immediately from  \cite[Proposition 5.8]{ps}. \end{proof}

\begin{proof}[Proof of Corollary~\ref{cor: incidence}] Let $W_{p,q}(\textbf{n})$ be the Stein filling of $(L(p,q), \xi_{can})$, and let $\mathcal{W}_\textbf{n}({\textbf{m}})$ be the (unbraided) wiring diagram constructed as in Definition~\ref{def: wd} based on the  blowup sequence $(0) \to (1,1) \to \cdots \to \textbf{n}$ and $\textbf{m} = \textbf{b}-\textbf{n}$.  Let $\G = \G_1 \cup \cdots \cup \G_n$ be the collection of symplectic graphical disks, with marked points $p_1, \ldots, p_m \in  \G$ as in Corollary~\ref{cor: arrangement}. Note that the incidence matrix $\mathcal{I}_\textbf{n}({\textbf{m}}): = \mathcal{I}(\G, \{p_j\})$ can be read off from the wiring diagram $\mathcal{W}_\textbf{n}({\textbf{m}})$, where each wire $w_i$ is identified with the disk $\G_i$ and the $x$- and $y$-type marked points are identified with $p_1, \ldots, p_m$.

To compute the incidence matrix  $\mathcal{I}_\textbf{n}({\textbf{m}})$, we enumerate the wires from top to bottom with respect to their {\em geometric} order on the right-hand side of the diagram $\mathcal{W}_\textbf{n}({\textbf{m}})$, and we enumerate the $x$- and $y$-type marked points using their natural geometric order from right to left. The matrix $\mathcal{I}_\textbf{n}({\textbf{m}})$ can be viewed as an extension of the incidence matrix $\mathcal{I}_\textbf{n}$  corresponding to the wiring diagram $\mathcal{W}_\textbf{n}$, which  carries only the $x$-type marked points. In the following, first we inductively construct $\mathcal{I}_\textbf{n}$ depending on the blowup sequence $(0) \to (1,1) \to \cdots \to \textbf{n}$, starting from the matrix
\[
\begin{blockarray}{cc}
 & x_1\\
\begin{block}{c(c)}
  w_1 & 0   \\
  w_2 & 1 \\
 \end{block}
\end{blockarray} \]  corresponding to the  wiring diagram with two wires $w_1, w_2$ and one marked point $x_1$, obtained by the initial  blowup $(0) \to (1,1)$. Suppose that we have constructed the $r \times (r-1)$ incidence matrix for a wiring diagram with $r$ wires and $r-1$ marked points $x_{1}, \ldots, x_{r-1}$. Assume that we insert the next wire $w_{r+1}$ into the diagram according to an exterior blowup. Then we ``blowup the incidence matrix" at hand by adding a last row and a last column which consist of all zeros except a $1$ at the lower right corner, to obtain the $(r+1) \times r$ incidence matrix.

Now assume that the last  wire $w_{r+1}$ is inserted into the diagram according to an interior blowup at the $i$th term. Then we blowup the incidence matrix at hand by {\em inserting}  a new row below the $(i+1)$st row and a last column so that the new row is copied from the row above and the last column is of the form $$[\underbrace{1,\ldots, 1}_{i}, 0,1, \underbrace{0, \ldots, 0}_{\geq 0}]^{T}$$ to obtain the $(r+1) \times r$ incidence matrix.

The matrix  $\mathcal{I}_\textbf{n}({\textbf{m}})$ can be obtained from $\mathcal{I}_\textbf{n}$ in a standard way based only on the $k$-tuple $\textbf{m}$. To extend $\mathcal{I}_\textbf{n}$ to $\mathcal{I}_\textbf{n}({\textbf{m}})$, for each $1 \leq s \leq k$, we just insert a column labelled with $y_s$, so that the first $s$ entries from the top of the column labelled with $y_s$ is $1$, and the rest are $0$. If $y_s$ has multiplicity $m_s$, then we repeat $m_s$-times the column labelled with $y_s$.  \end{proof}

Here we illustrate the proof of Corollary~\ref{cor: incidence} for the wiring diagram in Figure~\ref{fig: example2}. The incidence matrix $\mathcal{I}_\textbf{n}$ can be obtained algorithmically from the blowup sequence $(1,1) \to (1,2,1) \to (2,1,3,1) \to (2,1,4,1,2)=\textbf{n}$ used to construct  $\mathcal{W}_{\textbf{n}}$ as follows.

\[
\begin{blockarray}{cc}
 & x_1\\
\begin{block}{c(c)}
  w_1 & 0   \\
  w_2 & 1 \\
 \end{block}
\end{blockarray} \; \xrightarrow[\text{blowup}]{\text{exterior}}
\begin{blockarray}{ccc}
 & x_1 & x_2 \\
\begin{block}{c(cc)}
  w_1 & 0 & \blue{0} \\
  w_2 & 1 & \blue{0} \\
  w_3 & \blue{0} & \blue{1} \\
 \end{block}
\end{blockarray} \; \xrightarrow[\text{blowup}]{\text{interior}}
\begin{blockarray}{cccc}
 & x_1 & x_2 & x_3 \\
\begin{block}{c(ccc)}
  w_1 & 0 & 0 & \blue{1} \\
  w_2 & 1 & 0 & \blue{0} \\
  w_4 & \blue{1} & \blue{0} & \blue{1} \\
  w_3 & 0 & 1 & \blue{0} \\
 \end{block}
\end{blockarray} \; \xrightarrow[\text{blowup}]{\text{interior}}
\begin{blockarray}{ccccc}
 & x_1 & x_2 & x_3 & x_4\\
\begin{block}{c(cccc)}
  w_1 & 0 & 0 & 1& \blue{1} \\
  w_2 & 1 & 0 & 0 &\blue{1} \\
  w_4 & 1 & 0 & 1 & \blue{1} \\
  w_3 & 0 & 1 & 0& \blue{0} \\
  w_5 & \blue{0} & \blue{1} & \blue{0} & \blue{1}\\
 \end{block}
\end{blockarray}
 \]

The first arrow above corresponds to an exterior blowup, where we insert the last row and the last column which has a $1$ in the corner and $0$ everywhere else. The second arrow corresponds to an interior blowup at the first term, and we insert the third row and the last column so that the first two entries in the third row are copied form the row above and the entries in the last column are $1,0,1,0$ from the top. The last arrow corresponds to an interior blowup at the third term, and  we insert the fifth row and the last column so that the first two entries in the third row are copied form the row above and the entries in the last column are $1,1,1,0,1$ from the top.

Let $\textbf{m}=(0,1,1,1,1)$. To extend the $5 \times 4$ incidence matrix $\mathcal{I}_\textbf{n}$ for $\mathcal{W}_\textbf{n}$ to the $5 \times 8$ incidence matrix $\mathcal{I}_\textbf{n}({\textbf{m}})$ for  $\mathcal{W}_{\textbf{n}}(\textbf{m})$, we insert the columns labelled with $y_2, y_3, y_4, y_5$ where the first $s$ entries from the top of the column labelled with $y_s$ is $1$, and the rest are $0$.

\[
\begin{blockarray}{ccccccccc}
 & x_1 & x_2 & x_3 & x_4 & y_2  & y_3 & y_4 & y_5 \\
\begin{block}{c(cccc|cccc)}
  w_1 & 0 & 0 & 1 & 1 & 1 & 1 & 1 & 1  \\
  w_2 & 1 & 0 & 0 & 1 & 1 & 1 & 1 & 1 \\
  w_4 & 1 & 0 & 1 & 1 & 0 & 1 & 1 & 1  \\
  w_3 & 0 & 1 & 0 & 0 & 0 & 0 & 1 & 1 \\
  w_5 & 0 & 1 & 0 & 1 & 0 & 0 & 0 & 1 \\
\end{block}
\end{blockarray} = \mathcal{I}_\textbf{n}({\textbf{m}})
 \]

\begin{proof}[Proof of Corollary~\ref{cor: equiv}]  A matrix
$M$ with $r\geq 2$ rows $$v_i =(v_{ij}), \quad  v_{ij} \in \{0,1\}$$ is called a CQS matrix (see \cite[Definition 6.8]{djvs}) if $ \langle v_i, v_j \rangle = \langle v_i, v_i \rangle  -1$ holds for all  $1\leq i < j \leq r$, where $\langle \cdot,\cdot \rangle$ denotes the standard inner product, and CQS  stands for cyclic quotient singularity.

 Let $X$ be a cyclic quotient singularity whose singularity link is $(L(p,q), \xi_{can})$. According to  \cite[Theorem 6.18]{djvs}, there is a bijection between the smoothing components of $X$ and incidence matrices of the picture deformations of the decorated curve $(\mathcal C,l)$ with smooth branches representing $X$.  Moreover, the incidence matrices are in one-to-one correspondence with CQS matrices, up to permutations of the columns.

 Let $\mathcal{I}^R_{\textbf{n}}(\textbf{m})$ be the matrix obtained from $\mathcal{I}_{\textbf{n}}(\textbf{m})$ by reversing the order of the rows. As in
\cite[Lemma 6.11]{djvs}, one can check that  $\mathcal{I}^R_{\textbf{n}}(\textbf{m})$ is a CQS matrix and that there is a bijection between CQS matrices and the set of matrices of the form $\mathcal{I}^R_{\textbf{n}}(\textbf{m})$. The reason that we had to reverse the order of the rows of the matrix $\mathcal{I}_{\textbf{n}}(\textbf{m})$  is simply because  in the present paper, to construct $\mathcal{I}_{\textbf{n}}(\textbf{m})$,     we enumerated the wires in $\mathcal{W}_{\textbf{n}}(\textbf{m})$ from top to bottom, as opposed to bottom to top, with respect to their geometric order on the right-hand side of the diagram.

Finally, we give an explicit one-to-one correspondence between the Stein fillings of the contact singularity link $(L(p,q), \xi_{can})$ and the Milnor fibres of the associated cyclic quotient singularity. Let $W_{p,q}(\textbf{n})$ be the Stein filling given by Lisca obtained by the blowup sequence $(0) \to (1,1) \to \cdots \to \textbf{n}$.  Using the same blowup sequence and $\textbf{m} = \textbf{b} - \textbf{n}$, we can construct a CQS matrix  $\mathcal{I}^R_{\textbf{n}}(\textbf{m})$ as in the proof of  Corollary~\ref{cor: incidence}, which corresponds to a picture deformation, and hence gives a Milnor fibre as in \cite{djvs}. The Stein filling $W_{p,q}(\textbf{n})$ is diffeomorphic to the Milnor fibre, because while the  wiring diagram $\mathcal{W}_\textbf{n}({\textbf{m}})$ determines $W_{p,q}(\textbf{n})$, the picture deformation, which is in the same combinatorial equivalence class as the configuration of symplectic graphical disks arising from $\mathcal{W}_\textbf{n}({\textbf{m}})$, determines the Milnor fibre.

Conversely, for any Milnor fibre, which is obtained from a picture deformation whose incidence matrix is a CQS matrix, one can read off the pair $(\textbf{n}, \textbf{m})$ as in \cite[Proposition 6.12]{djvs}, and therefore construct the wiring diagram  $\mathcal{W}_\textbf{n}({\textbf{m}})$, so that the configuration of symplectic graphical disks arising from $\mathcal{W}_\textbf{n}({\textbf{m}})$, is in the same combinatorial equivalence class as the picture deformation.
\end{proof}

\begin{proof}[Proof of Corollary~\ref{cor: scott}] It is well-known that for any contact singularity link $(L(p,q), \xi_{can})$, the Milnor fibre of the Artin smoothing component of the corresponding cyclic quotient singularity gives the Stein filling $W_{(p, q)}((1,2,\ldots,2,1))$, which is Stein deformation equivalent to the one obtained by deforming the symplectic structure on the minimal resolution of the singularity; see \cite{bd}.

We describe a decorated germ $(\mathcal C,l)$ associated to the pair of coprime integers $p>q\geq 1$ such that $(\mathcal C,l)$
determines the cyclic quotient singularity with link $L(p,q)$ via the construction of de Jong and van Straten \cite{djvs}.
We follow the description given in \cite{npp}.
Suppose that $\frac pq=[a_1,\ldots,a_m]$. Let $G$ be the decorated linear graph having $m$ vertices $v_1,\ldots,v_m$ with the vertex $v_i$ weighted by the integer $-a_i$. Then $G$ is the dual resolution graph of a cyclic quotient
singularity with link $L(p,q)$. Let $G'$ be the simple graph obtained from $G$ by attaching $a_1-1$ new vertices to $v_1$
and $a_i-2$ new vertices to each vertex $v_i$ for $i>1$. Assign weight $-1$ to each new vertex. Finally, let $G''$
denote the graph obtained from $G''$ by endowing an arrowhead to each new vertex. Then $G''$ is an embedded
resolution graph of a germ of a plane curve $\mathcal C=\bigcup_i C_i$ with \emph{smooth} irreducible components corresponding to
the arrowheads in $G''$. The order of intersection of $C_i$ and $C_j$ corresponding to two distinct arrows is the number
of vertices on the intersection of the geodesics between the arrows and $v_m$. Also $l=(l_i)$, where the weight $l_i$ is the number of vertices on the geodesic
from the arrowhead corresponding to $C_i$ to $v_m$.

Now every continued fraction expansion $[c_1,\ldots,c_m]$, with $c_i\geq 2$ for all $i$, can be obtained from $[2]$ by repeated
applications of the following two operations:
\begin{enumerate}[(i)]
\item $[a_1,\ldots,a_r]\mapsto [a_1,\ldots,a_r+1]$,
\item $[a_1,\ldots,a_r]\mapsto [a_1,\ldots,a_r,2]$.
\end{enumerate}
Let $[b_1,\ldots,b_s]$ denote the dual string of $[a_1,\ldots,a_r]$, i.e. if $\frac pq=[a_1,\ldots,a_r]$, then
$\frac p{p-q}=[b_1,\ldots,b_s]$. Then the dual string changes for each of these two operations in the following way:
\begin{enumerate}[(i)]
\item $[b_1,\ldots,b_s]\mapsto [b_1,\ldots,b_s,2]$,
\item $[b_1,\ldots,b_s]\mapsto [b_1,\ldots,b_s+1]$.
\end{enumerate}
This follows immediately from a consideration of Riemenschneider's point diagrams (see \cite{r} or \cite{p}). We check that if the statement of
the corollary holds for a pair $(p,q)$ with $\frac pq=[a_1,\ldots,a_r]$, then it continues to hold if we replace
$(p,q)$ by $(p',q')$, where $\frac {p'}{q'}=[a_1,\ldots,a_r+1]$ or $\frac {p'}{q'}=[a_1,\ldots,a_r,2]$. As the statement
holds trivially if $\frac pq=[2]$, by induction, it will follow that the statement holds for every pair of coprime integers $p>q\geq 1$.

Suppose first that $\frac {p'}{q'}=[a_1,\ldots,a_r+1]$. Then we have $\frac {p'}{p'-q'}=[b_1,\ldots,b_s,2]$, where $[b_1,\ldots,b_s]=\frac {p}{p-q}$. Notice that in the wiring diagram for the planar Lefschetz fibration $$ W_{(p,q)}((1,2,\ldots,2,1)) \to D^2$$
constructed according to Theorem \ref{thm: main}, there is a final marked point $y_s$ through which all the wires
pass. It is easy to see that that the wiring diagram for the planar Lefschetz fibration
$W_{(p',q')}((1,2,\ldots,2,1))\to D^2$ is given by taking the wiring diagram for the planar Lefschetz fibration $W_{(p,q)}((1,2,\ldots,2,1)) \to D^2$ and inserting a new
wire $w_{s+1}$ at the bottom of the diagram on the right hand side, which runs under the wiring diagram  up to the point
that  it passes through the marked point $y_s$ (turning it into a $y_{s+1}$-type marked point) and then runs above the diagram. In addition, a single new marked point is placed on the new wire $w_{s+1}$.

Let $(\mathcal C,l)$ denote the decorated curve associated to the pair of integers $(p,q)$ constructed as above. Then
$\mathcal C$ will consist of $s$ curvettas by the induction hypothesis.
Let $(\mathcal C',l')$ denote the decorated curve associated to the pair $(p',q')$. Then it is easy to see that $\mathcal C'$ can
be obtained from $\mathcal C$ by adding an extra curvetta $C_{s+1}$ which has intersection multiplicity $1$ with each curvetta
$C_i,i\leq s$. In addition, $l'$ is given by setting $l'_i=l_i$ for $i\leq s$ and $l_{s+1}=2$. Since the Scott
deformation of $(\mathcal C,l)$ is combinatorially equivalent to the symplectic disk arrangement associated to the
wiring diagram for $W_{(p,q)}((1,2,\ldots,2,1))$ (by the induction hypothesis), it is easy to see that the Scott
deformation of $(\mathcal C',l')$ is combinatorially equivalent to the symplectic disk arrangement associated to the
wiring diagram for $W_{(p',q')}((1,2,\ldots,2,1))$.

Now suppose that $\frac {p'}{q'}=[a_1,\ldots,a_r,2]$. Then we have $\frac {p'}{p'-q'}=[b_1,\ldots,b_s+1]$.
In this case, the wiring diagram associated to $W_{(p',q')}((1,2,\ldots,2,1))$ can be obtained from the wiring diagram
associated to $W_{(p,q)}((1,2,\ldots,2,1))$ by adding one new simultaneous intersection point $y_s$ of all the wires.
Also the decorated curve $(\mathcal C',l')$ associated to the pair $(p',q')$ can be obtained from the decorated curve
$(\mathcal C,l)$ by increasing by $1$ the order of intersection between each pair of distinct curvettas $C_i$ and $C_j$
and by increasing each $l_i$ by $1$. Again it is easy to check that the Scott
deformation of $(\mathcal C',l')$ is combinatorially equivalent to the symplectic disk arrangement associated to the
wiring diagram for $W_{(p',q')}((1,2,\ldots,2,1))$.
\end{proof}

In the following we give an example to illustrate the proof of Corollary~\ref{cor: scott}. Note that the Stein  filling  $W_{(p, q)}((1,2,\ldots,2,1))$ is obtained via the sequence of {\em exterior} blowups $$(0) \to (1,1) \to (1,2,1) \to \cdots \to (1,2,\ldots,2,1),$$ and the corresponding unbraided wiring diagram can be drawn easily. The wiring diagram corresponding to $W_{(56, 17)}((1,2,2,2,1))$, for instance, is depicted in Figure~\ref{fig: example3}.

\begin{figure}[ht]  \relabelbox \small {\epsfxsize=4.3in
\centerline{\epsfbox{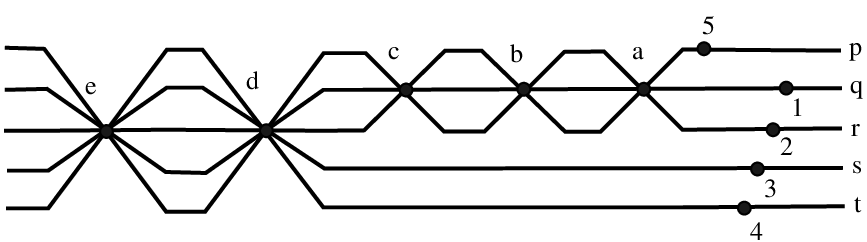}}} \relabel{a}{$y_3$} \relabel{b}{$y_3$} \relabel{c}{$y_3$} \relabel{d}{$y_5$}  \relabel{e}{$y_5$}  \relabel{3}{$x_3$} \relabel{2}{$x_2$} \relabel{1}{$x_1$}  \relabel{4}{$x_4$}    \relabel{5}{$y_1$} \relabel{p}{$w_1$} \relabel{q}{$w_2$} \relabel{r}{$w_3$} \relabel{s}{$w_4$}  \relabel{t}{$w_5$}
\endrelabelbox
\caption{Wiring diagram for the Stein filling $W_{(56, 17)}((1,2,2,2,1))$---the symplectic resolution.  }\label{fig: example3} \end{figure}

\begin{figure}[ht]  \relabelbox \small {\epsfxsize=5.7in
\centerline{\epsfbox{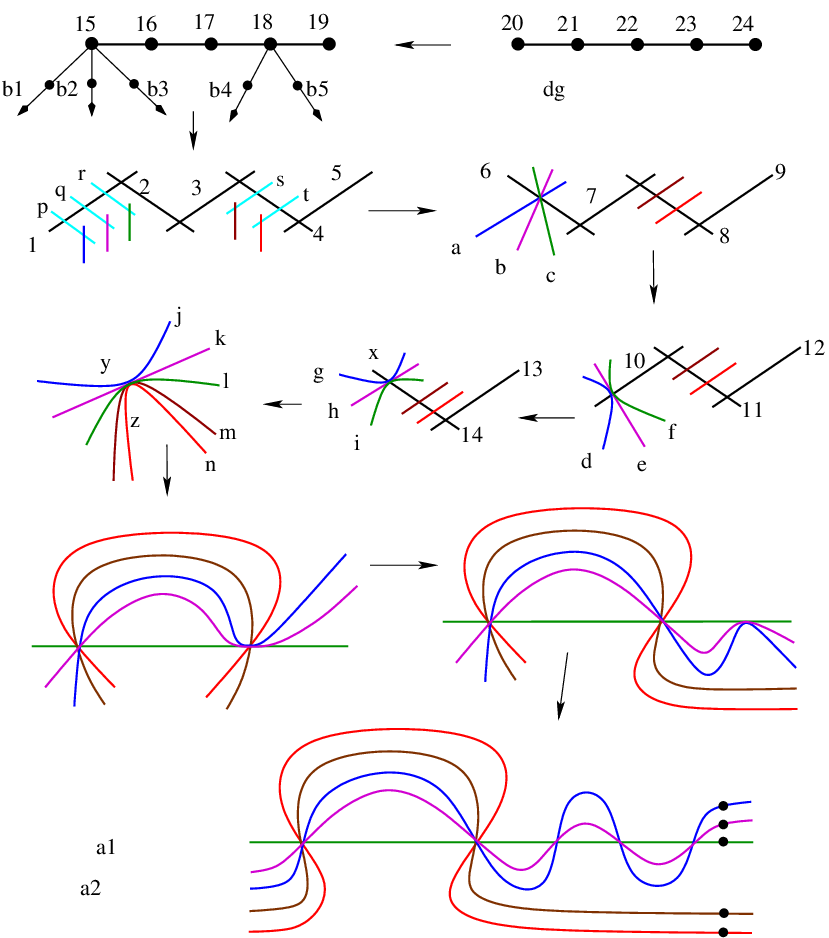}}} \relabel{3}{$-2$} \relabel{2}{$-2$} \relabel{1}{$-4$}  \relabel{4}{$-4$}    \relabel{5}{$-2$}  \relabel{7}{$-2$} \relabel{8}{$-2$}
 \relabel{9}{$-2$}  \relabel{6}{$-1$}   \relabel{11}{$-2$}
 \relabel{12}{$-2$}  \relabel{10}{$-1$} \relabel{13}{$-2$}  \relabel{14}{$-1$} \relabel{a}{$\boxed{2}$} \relabel{b}{$\boxed{2}$}  \relabel{c}{$\boxed{2}$}  \relabel{d}{$\boxed{3}$} \relabel{e}{$\boxed{3}$}  \relabel{f}{$\boxed{3}$}  \relabel{g}{$\boxed{4}$} \relabel{h}{$\boxed{4}$}  \relabel{i}{$\boxed{4}$} \relabel{j}{$\boxed{6}$} \relabel{k}{$\boxed{6}$}  \relabel{l}{$\boxed{6}$} \relabel{m}{$\boxed{3}$}  \relabel{n}{$\boxed{3}$} \relabel{p}{$-1$} \relabel{q}{$-1$} \relabel{r}{$-1$} \relabel{s}{$-1$}  \relabel{t}{$-1$} \relabel{x}{$\circled{3}$}  \relabel{y}{$\circled{5}$}   \relabel{z}{$\circled{2}$}  \relabel{15}{$-4$} \relabel{16}{$-2$} \relabel{17}{$-2$}  \relabel{18}{$-4$}    \relabel{19}{$-2$}  \relabel{20}{$-4$} \relabel{21}{$-2$}
 \relabel{22}{$-2$}  \relabel{23}{$-4$}   \relabel{24}{$-2$} \relabel{dg}{dual resolution graph} \relabel{a1}{Scott}  \relabel{a2}{deformation}
  \relabel{b1}{$-1$}  \relabel{b2}{$-1$} \relabel{b3}{$-1$} \relabel{b4}{$-1$} \relabel{b5}{$-1$} \endrelabelbox
\caption{From dual resolution graph of a cyclic quotient singularity to the Scott deformation of the decorated plane curve representing the singularity.  }\label{fig: scott} \end{figure}

 Next, we show that the arrangement of symplectic disks arising from the wiring diagram in Figure~\ref{fig: example3} is combinatorially equivalent  to the Scott deformation of $(\mathcal{C},l)$. First we note that $\frac{56}{17} = [4,2,2,4,2]$ and construct the curve $(\mathcal{C},l)$, as depicted in Figure~\ref{fig: scott}.

Note that $(\mathcal{C},l)$ consists of five curvettas with indicated weights and tangencies using the notation in \cite{ps}, where boxed numbers indicate the weights, circled numbers indicate the tangencies, and the other numbers are the self-intersection numbers of the rational curves.

We also depicted the Scott deformation of $(\mathcal{C},l)$ at the bottom in Figure~\ref{fig: scott}. It follows that, after enumerating the smooth branches of the Scott deformation of $(\mathcal{C},l)$ from top to bottom as they appear on the right-hand side,  the incidence matrix arising from  the Scott deformation coincides with the incidence matrix arising from the wiring diagram in  Figure~\ref{fig: example3}. Note that each branch of the Scott deformation $(\mathcal{C},l)$  includes a free marked point. The $5 \times 10$ incidence matrix $\mathcal{I}_{Artin} : =  \mathcal{I}_{\textbf{n}}(\textbf{m})$, where $\textbf{n}=(1,2,2,2,1)$ and $\textbf{m}=(1,0,3,0,2)$,  is given as follows.

\[
\begin{blockarray}{cccccccccccc}
 & x_1 & x_2 & x_3 & x_4 & y_1  & y_3 & y_3 & y_3 & y_5 & y_5\\
\begin{block}{c(cccc|ccccccc)}
  w_1 & 0 & 0 & 0 & 0 & 1 & 1 & 1 & 1 & 1 & 1 \\
  w_2 & 1 & 0 & 0 & 0 & 0 & 1 & 1 & 1 & 1 & 1 \\
  w_3 & 0 & 1 & 0 & 0 & 0 & 1 & 1 & 1 & 1 & 1 \\
  w_4 & 0 & 0 & 1 & 0 & 0 & 0 & 0 & 0 & 1 & 1 \\
  w_5 & 0 & 0 & 0 & 1 & 0 & 0 & 0 & 0 & 1 & 1 \\
\end{block}
\end{blockarray}= \mathcal{I}_{Artin}
 \]
Finally, one can easily obtain the {\em disjoint} vanishing cycles in $D_5$ of the planar Lefschetz fibration $W_{(56, 17)}((1,2,2,2,1)) \to D^2$, which we depicted in Figure~\ref{fig: ex4b}. As observed in \cite{ps}, this Lefschetz fibration agrees with the Lefschetz fibration constructed by Gay and Mark \cite{gm}, using the plumbing description $W_{(56, 17)}((1,2,2,2,1))$ given by the dual resolution graph.

\begin{figure}[ht]
  \relabelbox \small {\epsfxsize=1.8in
  \centerline{\epsfbox{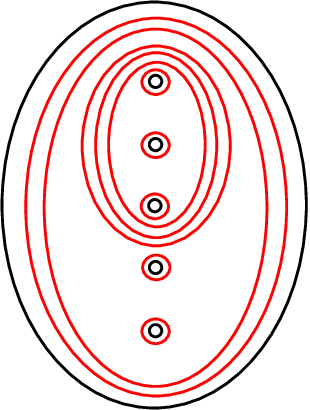}}}

\endrelabelbox
       \caption{The vanishing cycles of the planar Lefschetz fibration $W_{(56, 17)}((1,2,2,2,1)) \to D^2$.  }
        \label{fig: ex4b}
\end{figure}

Note that the total monodromy of the planar Lefschetz fibration $W_{(56, 17)}((1,2,2,2,1)) \to D^2$, the product of Dehn twists along the disjoint vanishing cycles depicted in Figure~\ref{fig: ex4b}, is the monodromy of the open book compatible with $(L(56,17), \xi_{can})$.  This monodromy has another positive factorization $$ D (\g_5)  D (\g_4) D (\g_3) D (\g_2)  D(\overline{\a}_4) D(\overline{\a}_3) D(\overline{\a}_2) D(\overline{\a}_1),$$ which is the total monodromy of the planar Lefschetz fibration $W_{(56, 17)}((2,1,4,1,2)) \to D^2$ we discussed in Section~\ref{subsec: exa}. \\

\end{document}